\documentclass[10pt, reqno]{amsart}
 
\usepackage{geometry}
\usepackage{amssymb}
\usepackage[utf8]{inputenc}
\usepackage{dsfont}
\usepackage[toc,page]{appendix}
\usepackage{accents}
\usepackage{xcolor}
\usepackage{stmaryrd}
\usepackage{esint}
\usepackage{todonotes} 
\usepackage{verbatim}
\usepackage{amsmath, tikz-cd}

\title[Discrete nonlocal quantum diffusion]{The spatially discrete to continuous limit in the nonlocal quantum diffusion equation}

\author{Daniel Matthes}
\address{Daniel Matthes -- Zentrum Mathematik, TU München, Boltzmannstrasse 3, D-85748 Garching, Germany}
\email{matthes@ma.tum.de}
\author{Eva-Maria Rott}
\address {Eva-Maria Rott -- Zentrum Mathematik, TU München, Boltzmannstrasse 3, D-85748 Garching, Germany}
\email{eva-maria.rott@tum.de}

\newcommand{\setR}{{\mathbb R}}
\newcommand{\setRp}{{\mathbb R}_{>0}}
\newcommand{\setRnn}{{\mathbb R}_{\ge0}}
\newcommand{\setC}{{\mathbb C}}
\newcommand{\setZ}{{\mathbb Z}}
\newcommand{\dd}{\,\mathrm{d}}
\newcommand{\dn}{\mathrm{d}}
\newcommand{\dq}{\mathrm{D}}
\newcommand{\crc}{{\mathbb T}} 
\newcommand{\crcN}{{\mathbb{T}_\delta}} 

\newcommand{\dent}{{\mathbb H}}
\newcommand{\loc}{\text{loc}}
\newcommand{\admA}{{\mathcal A}}

\newcommand{\one}{\mathbb{I}}

\newcommand{\eps}{\varepsilon}
\newcommand{\tr}[1]{\mathrm{Tr}\left\{#1\right\}}

\newcommand{\mm}{\mathcal{M}}

\newcommand{\dens}{\mathfrak{D}_\delta}
\newcommand{\densp}{\mathfrak{D}^+_\delta}
\newcommand{\densr}{\mathfrak{D}_\delta^\mathbb{R}}
\newcommand{\denspr}{\mathfrak{D}_\delta^{\mathbb{R},+}}

\newcommand{\commu}[2]{\left[#1,#2\right]}

\newcommand{\Lp}[1]{L^{#1}(\mathbb{T}_\delta)}
\newcommand{\cmp}[2]{\left\{#1\right\}_{#2}}

\newcommand{\eng}{\mathcal{E}_\delta}

\newcommand{\boltz}{{\mathcal H}}

\newcommand{\htkrnl}{{\mathfrak K}}
\newcommand{\heat}{{\mathcal K}}
\newcommand{\bnrm}[1]{\left\{\kern-0.25ex\left\{#1\right\}\kern-0.25ex\right\}}
\newcommand{\tnrm}[1]{{\left\vert\kern-0.25ex\left\vert\kern-0.25ex\left\vert #1 
    \right\vert\kern-0.25ex\right\vert\kern-0.25ex\right\vert}}
\newcommand{\adm}{{\mathcal A}}
\newcommand{\nop}{{\mathbf n}}

\newcommand{\dprb}{{\mathcal P}({\mathbb T}_\delta)}
\newcommand{\dprbp}{{\mathcal P}_{>0}({\mathbb T}_\delta)}
\newcommand{\wto}{\rightharpoonup}

\newtheorem{theorem}{Theorem}[section]
\newtheorem{proposition}[theorem]{Proposition}
\newtheorem{lemma}[theorem]{Lemma}
\newtheorem{corollary}[theorem]{Corollary}
\newtheorem{definition}[theorem]{Definition}
\newtheorem{remark}[theorem]{Remark}

\begin{document}
	
\begin{abstract}
  We propose and analyse a spatial discretization of the non-local Quantum Drift Diffusion (nlQDD) model by Degond, M\`{e}hats and Ringhofer in one space dimension.
  With our approach, that uses consistently matrices on $\setC^N$ instead of operators on $L^2$,
  we circumvent a variety of analytical subtleties in the analysis of the original nlQDD equation,
  e.g. related to positivity of densities or to the quantum exponential function.

  Our starting point is spatially discretized quantum Boltzmann equation with a BGK-type collision kernel,
  from which we derive the discretized nlQDD model in the diffusive limit.
  Then we verify that solutions dissipate the von-Neumann entropy, which is a known key property of the original nlQDD,
  and prove global existence of positive solutions, which seems to be a particular feature of the discretization.
  Our main result concerns convergence of the scheme:
  discrete solutions converge --- locally uniformly with respect to space and time --- to classical solutions of the the original nlQDD model on any time interval $[0,T)$ on which the latter remain positive.
  In particular, this extends the existence theory for nlQDD, that has been established only for initial data close to equilibrium so far.
\end{abstract}

\thanks{\emph{Acknowledgement:} This research has been supported by the DFG Collaborative Research Center TRR 109, ``Discretization in Geometry and Dynamics.''}

\maketitle

	
\section{Introduction}
\subsection{The nlQDD model in brief}
\label{sct:brief}
Quantum Drift Diffusion (QDD) equations have been derived in the context of semi-conductor modeling and plasma physics as improvement over purely classical drift-diffusion equations for the charge carriers by addition of quantum mechanical corrections, see e.g. \cite{MRS} for an overview. Several such quantum corrected models are available in the literature, here we are concerned with the particular one proposed by Degond, M\`{e}hats, and Ringhofer in \cite{DeMeRi05}; see Sections \ref{sct:moments} below for a brief review of its derivation. A characteristic feature of this model is the non-locality in space of the governing PDE; we shall refer to it as \emph{non-local} Quantum Drift Diffusion (nlQDD) for definiteness.

Specifically, we study nlQDD --- and its spatial discretization --- on the one-dimensional torus $\crc$, i.e., subject to periodic boundary conditions, without external or self-consistent electric potentials. The sought time-dependent density $n\in L^1_{\ge0}(\crc)$ of the quantum particles satisfies 
\begin{align}
    \label{eq:QDD}
  \partial_t n = \partial_x (n \,\partial_x A)
  \quad \text{subject to}\ n=\nop[A].
\end{align}
To determine the potential $A:\crc\to\setR$ from the constitutive relation $n=\nop[A]$ involves the solution of the following inverse problem: find $A$ such that the operator exponential of the Schr\"odinger operator $-\hbar^2\Delta-A$ on $L^2(\crc)$ has the density $n$ on its diagonal. More precisely, if $\mathfrak K_A\in C(\crc\times\crc)$ denotes the integral kernel for the trace class operator $\exp(\hbar^2\Delta+A)$, then the relation $n=\nop[A]$ means that $n(x)=\mathfrak K_A(x,x)$ for all $x\in\crc$. The operator $\mm[n]=\exp(\hbar^2\Delta+A)$ itself is referred to as \emph{local quantum Maxwellian}, see Section \ref{sct:moments} below for more details.
In the (semi-classical) limit $\hbar\searrow0$, the condition $n=\nop[A]$ simplifies to the pointwise relation $n(x)=e^{A(x)}$, and \eqref{eq:QDD} collapses to the linear heat equation,
\begin{align*}
  \partial_t n = \partial_x(n\,\partial_x\log n) = \partial_{xx}n.
\end{align*}
However, for fixed $\hbar>0$, which is of interest in the paper at hand, the operator $\nop$ is non-local in space, which makes \eqref{eq:QDD} both non-linear and non-local.
In this paper, we shall make a conceptually novel approach to the solution of \eqref{eq:QDD} via spatial discretization, thereby extending the available existence theory, which is currently limited to initial conditions close to equilibrium, see Section \ref{sct:history} for more details.

\subsection{Spatial discretization and main results}
For spatial discretization of \eqref{eq:QDD}, we consider a uniform mesh $\crcN$ with $N$ cells, each of width $\delta=1/N$. With $\dq_\delta^\pm$ being the forward/backward difference quotients and $\Delta_\delta=\dq_\delta^+\dq_\delta^-$ the symmetric discrete Laplacian --- see Section \ref{sct:basics} for their definitions --- the discrete nlQDD equation for a piecewise constant w.r.t. $\crcN$ density $n_\delta$ is given by
\begin{align}
  \label{eq:DQDD0}\tag{$\text{nlQDD}_\delta$}
  \dot n_\delta = \dq_\delta^-\big( \nu_\delta^+\,\dq_\delta^+A_\delta\big),
  \quad \text{subject to}\ n_\delta=\nop_\delta[A_\delta].
\end{align}
The constitutive relation $n_\delta=\nop_\delta[A_\delta]$ is analogous to the one in \eqref{eq:QDD}: any $A_\delta:\crcN\to\setR$ gives rise to a discrete Schr\"odinger operator $-\hbar^2\Delta_\delta-A_\delta$, i.e., a symmetric matrix in $\setR^{N\times N}$. The inverse problem amounts to determine $A_\delta$ such that the diagonal entries of the matrix exponential $M=\exp(\hbar^2\Delta_\delta+A_\delta)$ are $n_\delta$, i.e., $\delta n(j\delta)=M_{jj}$; the values $\nu_\delta^+$ in \eqref{eq:DQDD0} are given by the next-to-digonal entries, $\delta\nu_\delta^+(j\delta)=M_{j,j+1}$. The matrix exponential $M$ itself, which we call \emph{discrete quantum Maxwellian} $\mm_\delta[n_\delta]$, serves as a spatially discrete surrogate for $\mm[n]$ above.

We shall prove three main results about \eqref{eq:DQDD0} and its implications on \eqref{eq:QDD}, which we summarize below. We use the following notations: $\dprbp$ is the set of \emph{positive probability densities} $n:\crcN\to\setRp$, i.e., $\delta\sum_{\xi\in\crcN}n(\xi)=1$, the set $\densp$ of \emph{postitive density matrices} contains the self-adjoint positive definite $R\in\setC^{N\times N}$ with unit trace. Each $R\in\densp$ gives rise to a $n\in\dprbp$ via
\begin{align}
  \label{eq:R2n}
  \delta\,n(j\delta)=R_{jj} \qquad \text{for $j=1,\ldots,N$}.  
\end{align}
Our first main result concerns the derivation of \eqref{eq:DQDD0} from a spatially discrete quantum kinetic model, which parallels \cite{DeRi03,DeMeRi05}. Specifically, we consider the rescaled \emph{discrete collisional quantum Liouville equation}
\begin{align}
  \label{eq:DQL0}
  \eps \dot R^\eps = i\hbar\,\commu{\Delta_\delta}{R^\eps} - 2\frac{R^\eps-\mm_\delta[n^\eps]}\eps,
\end{align}
which is the straight-forward discretization of the original quantum Liouville equation \eqref{eq:qLiouville} that is recalled in Section \ref{sct:qLiouville} below. The following theorem summarizes Proposition \ref{prop:DQD}, and further elements of Section \ref{sct:QL}, see in particular
\begin{theorem}
  \label{thm:1}
  Fix a discretization $\delta=1/N>0$ and an initial density matrix $R^\text{in}\in\densp$. For each $\eps>0$, let $R^\eps:\setRnn\to\densp$ be the (global unique) solution to \eqref{eq:DQL0} with $R^\eps(0)=R^\text{in}$, and let $n^\eps:\setRnn\to\dprbp$ be the density according to \eqref{eq:R2n}. 

  Then $n^\eps$ converges locally uniformly in time to a solution $n_\delta$ of \eqref{eq:DQDD0},  as $\eps\searrow0$. The initial data $n^\text{in}=n_\delta(0)$ and $R^\text{in}$ are related by \eqref{eq:R2n}.
\end{theorem}
The theorem above essentially justifies to consider \eqref{eq:DQDD0} as ``the right'' discretization of \eqref{eq:QDD}.
Our second main result is about qualitative properties of solutions to \eqref{eq:DQDD0}. The following theorem summarizes the two propositions from Section \ref{sct:nlQDD}.
\begin{theorem}
  \label{thm:2}
  Fix a discretization $\delta=1/N>0$ and an initial density $n_\delta^0\in\dprbp$

  Then there exists a unique and global positive solution $n_\delta:\setRnn\to\dprbp$ to \eqref{eq:DQDD0}. Moreover, the following entropy functional is monotone in time:
  \begin{align}
    \label{eq:dent0}
    \dent(n_\delta) = \delta\sum_{\xi\in\crcN}n_\delta(\xi)A_\delta(\xi).
  \end{align}
\end{theorem}
A key point above is that the discrete solutions $n_\delta$ provide a strictly positive approximation.
The third and final result is the convergence of the discretization. The following theorem summarizes Propositions \ref{prp:nconv}, \ref{prp:abstract} and \ref{prp:theendofQDD} from Section \ref{sct:convergence}.
\begin{theorem}
  \label{thm:3}
  Fix a positive initial datum $n^\text{in}\in L^1(\crc)$ of unit mass and finite Fisher information. Let a sequence of discretizations $\delta=1/N$ going to zero be given. Define discretized initial data $n_\delta^\text{in}\in\dprbp$ by averaging on cells, and let $n_\delta:\setRnn\to\dprbp$ be the global solutions to \eqref{eq:DQDD0}.

  Then there is some $T_*>0$ such that the $n_\delta$ converge, locally uniformly on $[0,T_*)\times\crc$, to a positive solution $n_*:[0,T_*)\times\crc\to\setRp$ of \eqref{eq:QDD} with $n_*(0)=n^\text{in}$. Moreover, the discrete potentials $A_\delta$ converge in $L^2([0,T_*),H^1(\crc))$ to the corresponding potential $A_*$. Finally, either $T_*=+\infty$, or $\inf n_*(t;\cdot)\to0$ as $t\nearrow T_*$. 
\end{theorem}

\subsection{Origin of nlQDD: moment hydrodynamics}
\label{sct:moments}
The nlQDD model is a cornerstone of \emph{quantum moment hydrodynamics}. The latter is a framework for the derivation of effective macroscopic evolution equations for interacting many particle quantum systems from first principles. The approach gave rise to quantum mechanical corrected versions e.g. of classical energy transport \cite{DeMeRi05} and hydrodynamical \cite{AnsgarQHD} models, see \cite{jungel12} for a review; nlQDD itself arises as a corrected version of the classical drift diffusion equation.

Quantum moment hydrodynamics parallels the classical moment closure method of hydrodynamics by Levermore, see \cite{Levermore96}, where a closed system of evolution equations for certain moments --- particle number density, energy density etc. --- of solutions to a Boltzmann equation is derived by assuming that the particle distribution $f(x,p)$ on phase space is always close to the manifold of local equilibria. This manifold consists of local Maxwellians that are minimizers of the Boltzmann entropy subject to matching the current moments under consideration.
In the simplest situation in which just the particle number density $n(x)=\int f(x,p)\dd p$ needs to be met, and the relative Boltzmann entropy simplifies to $\mathbf{H}_\text{Boltz}(f)=\int f(x,p)\big(\log f(x,p)+|p|^2\big)\dd p\dd x$, the local Maxwellians have the form $M[n](x,p)=\exp(-|p|^2+a(x))$, with Lagrange multiplier $a(x)=\log(n(x)/N)$ and normalization constant $N$. In the diffusive limit (moving to a time scale and collision frequency of the same large order) of the Boltzmann equation with BGK collisional operator, one obtains the linear diffusion equation $\partial_tn=\Delta n$ as effective model.

The approach in \cite{DeRi03} is a quantum mechanical generalization of this, using density operators in place of densities on phase space, and a quantum Boltzmann equation, see \eqref{eq:qLiouville} below, in place of the classical one. The role of $f(x,p)$ is then played by a density operator $\rho$, or equivalently, by its Hilbert-Schmidt kernel $\mathfrak K$. 
In analogy to the classical case, the manifold of local equilibria consists of local quantum Maxwellians, which are minimizers of the relative von Neumann entropy, keeping the relevant moments fixed.
In the situation of nlQDD, with the only moment condition on the spatial particle density $n(x)=\mathfrak K(x,x)$, the relative von Neumann entropy simplifies to $\mathbf{H}_\text{vN}(\rho)=\tr{\rho(\log\rho-\hbar^2\Delta)}$, and one finds in analogy to the functions $M[n]$ above local quantum Maxwellians as operator exponentials $\mm[n]=\exp(\hbar^2\Delta+A)$. The diffusive limit in the quantum Boltzmann equation with BGK collisions, see \eqref{eq:qLiouville} below, yields the QDD model \eqref{eq:QDD}.

The main difference between the classical and the quantum approach is that the quantum mechanical Lagrange multiplier $A$ is no longer an explicit pointwise expression in terms of $n$, but involves the solution of the inverse problem $n=\nop[A]$ that has been introduced in Section \ref{sct:brief} above. That inverse problem has been addressed in a series of papers: \cite{MePi10,MePi11,MePi17,MePi19}. By a semi-classical expansion, neglecting higher powers of the Planck constant, the problem becomes local in space, but of fourth or even higher order \cite{DeMeRi05,JM-ent}. The non-locality apparently changes the qualitative behavior of solutions, which can be observed numerically, see e.g. \cite{GaMe05}. We emphasize that qunatum mechnically induced non-locality is actually not only a challenge for the analysis, but also has beneficial effects, like providing a time-uniform control of the solution in $H^1$.

\subsection{Derivation of nlQDD: quantum Boltzmann equation}
\label{sct:qLiouville}
As outlined above, the nlQDD equation \eqref{eq:QDD} is obtained in the diffusive limit of the following quantum Boltzmann equation with BGK collision kernel:
\begin{equation}\label{eq:qLiouville}
	i\hbar\partial_t \rho = \commu{-\frac{\hbar^2}2\Delta}{\rho} - \frac{\rho-\mm[\rho]}\tau.
\end{equation}
Above $\rho$ is the density operator for a single-particle state, the reversible part of the dynamics is given by the free Hamiltonian $H_0=-\hbar^2/2\Delta$, and the effect of many-particle interactions are modelled only by means of the BGK relaxation term, with relaxation time $\tau>0$. Here $\mm[\rho]$ is the local quantum Maxwellian that has the same particle density as $\rho$. The evolution equation \eqref{eq:qLiouville}, that combines Schr\"odinger dynamics with entropy minimization, is entirely based on \emph{first principles} of many-particle quantum dyanmics. In \cite{MePi17}, existence of global classical solutions to \eqref{eq:qLiouville} is shown under restrictive assumptions on the initial data. The diffusive limit towards \eqref{eq:QDD} is performed by choosing $\tau=\eps/\hbar$ for a small parameter $\eps>0$ in \eqref{eq:qLiouville}, and rescaling time $t'=2\eps t$ to focus on the long-time dynamics. It has been formally shown in \cite{DeMeRi05} that solutions $\rho^\eps$ to the rescaled equation converge ``on their diagonal'' to solutions of nlQDD in the diffusive limit $\eps \to 0$.

In an analogous manner, we derive our spatially discrete nlQDD model \eqref{eq:DQDD0} from the direct spatial discretization \eqref{eq:DQL0} of \eqref{eq:qLiouville}, obtained by replacing the density operator $\rho$ on $L^2(\crc)$ by a density matrix $R\in\setC^{N\times N}$, the discrete Laplace operator by its discretization by central finite differences, and the local Maxwellian $\mm[\rho]$ by its discrete counterpart $\mm_\delta[R]$, which is the minimizer of the (discrete) relative von Neumann entropy. Thus our discrete quantum Boltzmann equation \eqref{eq:DQL0} results from discretization of the same \emph{first principles} that originally lead to \eqref{eq:qLiouville}. Moreover, Theorem \ref{thm:1} shows that our discrete nlQDD equation \eqref{eq:DQDD0} appears as an effective model in the diffusive limit of \eqref{eq:DQL0}, thus justifying this equation beyond an ad hoc discretization of \eqref{eq:QDD}. In the paper at hand, we thus follow the solid arrows in Figure \ref{fig:diagram}.

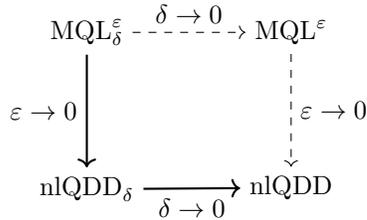
\begin{figure}[h]
  \centering
  \begin{tikzpicture}[node distance=1.5cm, auto]
    \node (Pe) {$\textnormal{MQL}^\eps_\delta$};
    \node (P0) [right=of Pe] {$\textnormal{MQL}^\eps$};
    \node (PeDt) [below=of Pe] {\ref{eq:DQDD0}};
    \node (P0Dt) [below=of P0] {nlQDD};
    \draw[->, dashed] (Pe) -- node[above] {$\delta \to 0$} (P0);
    \draw[->, thick] (PeDt) -- node[below] {$\delta \to 0$} (P0Dt);
    \draw[->, thick] (Pe) -- node[left] {$\varepsilon \to 0$} (PeDt);
    \draw[->, dashed] (P0) -- node[right] {$\varepsilon \to 0$} (P0Dt);
  \end{tikzpicture}
  \caption{Our passage from the discrete collisional quantum Liouville equation to nlQDD: we first perform the diffusive limit, resulting in a spatially discrete nlQDD equation, and then pass to the continuous limit.}
  \label{fig:diagram}
\end{figure}

\subsection{Previous analytical results and novel ideas}
\label{sct:history}
The first and so far only proof for global existence of solutions to \eqref{eq:QDD} has been given by Pinaud in \cite{Pinaud19}, using implicit discretization in time. The result is conditional on the solution being close to equilibrium. Specifically, it is required that the free energy is sufficiently small initially --- and hence, by energy monotonicity in time, also for all subsequent time steps --- to imply a time-uniform positive lower bound $\underline n>0$ on the solution. That positivity then allows to pass to the time-continuous limit. Our result in Theorem \ref{thm:3} is a significant extension of that.

While comparatively little is known about the full non-local model \eqref{eq:QDD}, the semi-classical approximations of \eqref{eq:QDD} of fourth and also sixth order have been studied exhaustively in the literature. The fourth order approximation is the so-called \emph{DLSS equation},
\begin{align}
  \label{eq:DLSS}
  \partial_tn = -\partial_{xx}\big(n\,\partial_{xx}\log n\big).
\end{align}
Equation \eqref{eq:DLSS} first appeared as a scaling limit in the Toom spin model \cite{DLSS1,DLSS2}. Global existence and uniqueness of weak and classical solutions, convergence to equilibrium, and self-similar long-time asymptotics have been studied by various authors, also in higher dimensions, see e.g. \cite{BLS,F,JP1,JP2,CT,JM-DLSS,GST,MMS}; see also \cite{MRSS} for a recent review on numerical methods for \eqref{eq:DLSS}.
The semi-classical approximation of order six in space is given by 
\begin{align}
  \label{eq:QDD6}
  \partial_tn = \partial_x\left[n\,\partial_x\left(\frac{\partial_{xx}(n\,\partial_{xx}\log n)}{n}+\frac12\big(\partial_{xx}\log n\big)^2\right)\right].
\end{align}
Global existence in dimensions up to three has been shown \cite{BJM,MR}; some information on self-similar long-time asymptotics is available as well \cite{MR}.

Apparently, the methods used in the analysis of local approximations are only of limited use for the analysis of the full non-local model \eqref{eq:QDD}.
The two key ingredients in the global existence analysis of both \eqref{eq:DLSS} or \eqref{eq:QDD6} are that, first, the right-hand sides of these equations can be re-written entirely as a sum of products of spatial derivatives of certain roots $\sqrt[r]n$ of the solution, and that, second, a priori estimates on these roots follow from dissipation of Boltzmann's $H$-functional $\boltz[n]=\int n\log n\dd x$ in combination with suitable non-linear functional inequalities. In contrast, for the full non-local model \eqref{eq:QDD}, first, no re-writing of $n\,\partial_xA$ as the product of ($x$-derivatives of) ``simple'' functions in $n$ is known, and second, entropy dissipation does apparently not provide any useful a priori estimates when the solution touches zero.

The novelty of the approach taken here is to perform the entire analysis of \eqref{eq:QDD} --- derivation from a kinetic model, a priori estimates, existence and positivity --- consistently on grounds of the spatial discretization \eqref{eq:DQDD0}. By means of finite dimensionality, we circumvent various subtle functional analytical challenges that arise already at the beginning of the treatment of \eqref{eq:QDD}, first and foremost the questions around unique solvability of the inverse problem $n=\nop[A]$. Indeed, the problem's discrete analogue $n_\delta=\nop_\delta[A_\delta]$ reduces to a convex optimization on $\setR^N$, for which existence and uniqueness of a solution are clear. The eventual passage to the continuum limit $\delta\searrow0$ requires mesh-independent a priori estimates: formally, we use the same entropy dissipation estimates as in the continuous situation \cite{Pinaud19}, but again, the manipulations are more easily justified on the discrete level.

Another conceptual advantage of the discrete approach is the preservation of positivity. The original nlQDD model \eqref{eq:QDD} breaks down after the first time at which the solution loses strict positivity, simply because the rigorous definition of $n=\nop[A]$ is limited to positive left-hand sides. For solutions to \eqref{eq:QDD} with general initial data, loss of positivity could not be excluded so far. On the other hand, we show in Theorem \ref{thm:2} that our discretization \eqref{eq:DQDD0} guarantees globally positive solutions $n_\delta$. Moreover, the $n_\delta$ converge, as $\delta\searrow0$, locally uniformly to a limit curve $n_*:[0,\infty)\to L^1(\crc)$, and $n_*$ satisfies \eqref{eq:QDD} whenever it is positive on $\crc$. Thus the discretization provides a solution concept for \eqref{eq:QDD} that works even beyond the loss of strict positivity.

\subsection{Plan of the paper}
After introducing notations, mainly related to the discretization, in Section \ref{sct:basics}, we prove for the spatially discrete quantum Liouville equation \eqref{eq:DQL0} global existence and positivity of solutions, and convergence to solutions of the discrete nlQDD \eqref{eq:DQDD0} in Section \ref{sct:QL}, resulting in Theorem \ref{thm:1}. Next, in Section \ref{sct:nlQDD}, we establish entropy dissipation and positivity of solutions to \eqref{eq:DQDD0}, thus proving Theorem \ref{thm:2}. In Section \ref{sct:hell}, an auxiliary result on quantitative approximation of quantum Maxwellians by their discrete counterparts is obtained, which is a key technical ingredient for the proof of the main convergence result from Theorem \ref{thm:3} in Section \ref{sct:convergence}.

\section{Notations and Basics}
\label{sct:basics}

\subsection{Functions, spaces, operators}
The physical space for the models under consideration is the one dimensional torus $\crc:=\setR/\setZ\cong(0,1]$.
The function spaces $L^2(\crc)$, $H^1(\crc)$ etc.\ are introduced in the usual way. Hilbert-Schmidt operators $\mathrm{L}$ on $L^2(\crc)$ are identified with their kernels $\mathfrak{K}\in L^2(\crc\times\crc)$ via
\begin{align*}
  \big(\mathrm{L}\psi\big) (x)  = \int_\crc \mathfrak{K}(x,y)\psi(y)\dd y \qquad \text{for all $\psi\in L^2(\crc)$.}
\end{align*}
The periodic Laplacian acting on $L^2(\crc)$ is denoted by $\Delta$, i.e., $\Delta f=\partial_{xx}f$.

\subsection{Spatial discretization}
For discretization of $\crc$ by $N\in\{2,3,4,\ldots\}$ cells of length $\delta:=1/N$, introduce the mesh $\crcN:=\delta(\setZ/N\setZ)\cong\{\delta,2\delta,\ldots,(N-1)\delta,1\}$ consisting of $N$ uniformly distributed sites $\xi\in\crc$. To each $\xi\in\crcN$, we associate the interval $I_\xi:=(\xi-\delta/2,\xi+\delta/2)\subset\crc$ of length $\delta$ centered around $\xi$, modulo translation by integer multiples of $\delta$. Each $f:\crcN\to\setC$ then has a piecewise constant representative $F:\crc\to\setC$, with $F(x)=f(\xi)$ for all $x\in I_\xi$. We define the integral of $f$ by the integral of $F$, and denote it by
\begin{align*}
  \delta\sum_\xi f := \delta\sum_{\xi\in\crcN}f(\xi) = \int_{\crc} F(x)\dd x .
\end{align*}
In general, we distinguish between vectors $v\in\setC^N$ and functions $f:\crcN\to\setC$. For the former, we use the euclidean norm $|v|=(|v_1|^2+\cdots+|v_n|^2)^{1/2}$, for the latter the following spatially discrete $L^p$-norms, which are natural for the eventual passage to the continuuum limit:
\[
  \|f\|_{\Lp p}= \bigg(\delta\sum_{\xi}|f|^p\bigg)^{1/p},
  \quad \text{for $1\le p<\infty$, and} \quad
  \|f\|_{\Lp\infty}=\max_{\xi}|f|.
\]
These definitions are consistent with the respective $L^p(\crc)$-norms of $f$'s piecewise constant representation $F\in L^p(\crc)$, i.e., Note that the spaces $\Lp p$ are all isomorphic to $\setC^N$, they only differ in their norms. $\Lp 2$ becomes a Hilbert with respect to the scalar product
\begin{align*}
  (f,g) = \delta\sum_{\xi}\bar f\, g.
\end{align*}
As a subset of $\Lp 1$, we single out the convex sets of non-negative and of positive unit-mass densities,
\begin{align*}
  \dprb &= \left\{ f:\crcN\to\setRnn\,\middle|\,\delta\,\textstyle{\sum_{\xi\in\crcN}} f(\xi) = 1\right\}, \\
  \dprbp &= \left\{ f\in\dprb\,\middle|\,\text{$f(\xi)>0$ for all $\xi\in\crcN$}\right\}.
\end{align*}

\subsection{Discrete operators, matrices, densities}
Linear operators $\operatorname L:\Lp p\to\Lp p$ are canonically identified with matrices $A\in\setC^{N\times N}$,
\begin{align*}
  \big(\operatorname L f\big)(j\delta) = \sum_{k=1}^NA_{jk}f(k\delta) \qquad \text{for all $f\in\Lp p$.}
\end{align*}
On the other hand, since any  $\operatorname L:\Lp p\to\Lp p$ is automatically Hilbert Schmidt, it possesses an integral kernel $\mathfrak K_\delta:\crcN\times\crcN\to\setC$, i.e.,
\begin{align*}
  \big(\operatorname L f\big)(\xi) = \delta\sum_{\eta\in\crcN}\mathfrak K_\delta(\xi,\eta)f(\eta) \qquad \text{for all $f\in\Lp p$.}
\end{align*}
These two representations of $\operatorname L$ are simply related by
\begin{align}
  \label{eq:AK}
  A_{jk} = \delta\mathfrak K_\delta(j\delta,k\delta).
\end{align}
Unfortunately, we need both representations in the following, depending on the context; we use roman letters for matrices and frakture letters for kernels. The matrix representation will be used when it comes to eigenvalues and functions of $\operatorname L$, e.g., the matrix of the operator exponential $\exp(\operatorname L)$ is simply the matrix exponential $\exp(A)$. The kernel representation is more natural in the analysis of the continuum limit, e.g., in the proof of convergence of the particle density that sits on the diagonal of resolvent operators, $n(\xi)=\mathfrak K_\delta(\xi,\xi)$. 

Further, we introduce the manifold $\dens$ of \emph{discretized density operators} as the set of self-adjoint positive semi-definite matrices with unit scaled trace,
\begin{align}
  \label{eq:densmat}
  \dens := \left\{ R \in \setC^{N \times N}\, \middle|\, R = R^*,\, \tr{R} = 1,\, R \geq 0 \right\}.
\end{align}
The subset of real matrices is denoted by $\densr = \dens \cap \setR^{N \times N}$, and the respective subsets of positive definite matrices by $\densp$ and $\denspr$.

\subsection{Difference operators}
Introduce the right and left difference quotient operators $\dq_\delta^+$ and $\dq_\delta^-$, respectively, by
\begin{align*}
    \dq_\delta^+f(x) = \frac1{\delta}\big(f(x+\delta)-f(x)\big),
    \quad
    \dq_\delta^-f(x) = \frac1{\delta}\big(f(x)-f(x-\delta)\big),
\end{align*}
as well as the central discretized Laplacian $\Delta_\delta$ by
\begin{align*}
  \big(\Delta_\delta f\big)(x) = \dq_\delta^+\dq_\delta^-f(x) = \dq_\delta^-\dq_\delta^+f(x) = \frac1{\delta^2}\big(f(x+\delta)+f(x-\delta)-2f(x)\big),
\end{align*}
Note that $\dq_\delta^\pm$ and $\Delta_\delta$ are well-defined operations both on $f\in L^p(\crc)$ and $f\in \Lp p$.
Further, the different quotient operators are mutually adjoint with respect to $L_\delta^2(\crc)$, i.e., $(\dq^\pm)^T=-\dq^{\mp}$.
%

 
\section{Discrete quantum Liouville equation and its diffusive limit}
\label{sct:QL}
In this section, we shall define a spatially discrete analogue of the modified quantum Liouville equation \eqref{eq:qLiouville} on density matrices $R\in\dens$, see \eqref{eq:densmat}, prove the well-definedness of the initial value problem, and perform the diffusive limit to the spatially discrete nlQDD.

\subsection{Quantum Entropy Functional}
Before introducing the discretized quantum Liouville equation, we define quantum Maxwellians. 
\begin{definition}
  The \emph{discrete free energy} $\eng[R]$ of a density matrix $R \in \dens$ is defined by
  \begin{equation}\label{eq: relenerg}
    \eng[R] = \tr{R(\log R - \hbar^2 \Delta_\delta)}.
  \end{equation}
  A \emph{discrete quantum Maxwellian} $\mm\in\dens$ for a given $n \in\dprbp$ is a solution of the constrained minimization problem
  \begin{equation}\label{eq:minprob}
    \eng[R]\to\min \quad \text{subject to} \quad R \in \dens \ \text{with}\ R_{kk}= \delta\,n(k\delta)\ \text{for all $k=1,2,\ldots,N$}.
  \end{equation}
\end{definition}
The inconvenient factor $\delta$ inside the constraint in \eqref{eq:minprob} is explained from the scaling \eqref{eq:AK}: for consistency with the continuum limit, the values of the particle density $n$ are identified with the diagonal of $R$'s Hilbert-Schmidt kernel, not with $R$'s diagonal.
\begin{proposition}[Discrete quantum Maxwellian]\label{th:quantmax}
  For any $n\in\dprbp$, the problem \eqref{eq:minprob} has a unique solution, given by
  \begin{equation}\label{eq:qM}
    \mm_\delta[n] := \exp (\hbar^2\Delta_\delta+A),
  \end{equation}
  with a suitable Lagrange multiplier $A\in\Lp\infty$, understood as a diagonal matrix in \eqref{eq:qM} above.
  In particular, the Maxwellian is real and positive definite, $\mm_\delta[n]\in\denspr$,
  and the map $\mm_\delta$ is injective from $n\in\dprbp$ to the Lagrange multipliers $A\in\Lp\infty$.
\end{proposition}
Before proceeding to the proof, we fix some notations for later reference.
\begin{definition}
  \label{dfn:nop}
  If $A\in\Lp\infty$ is the Lagrange multiplier for $n\in\dprbp$ in \eqref{eq:minprob}, then we call $n$ the \emph{quantum exponential} of $A$ and write $n=\nop_\delta[A]$.
  Moreover, for $R\in\densp$, we define $R$'s \emph{projection to the discrete quantum Maxwellians} by $\mm_\delta[R]:=\mm_\delta[n^R]$, with $n^R\in\dprbp$ given by $\delta\,n^R(k\delta)=R_{kk}$.
\end{definition}
\begin{remark}
  Clearly, $\nop_\delta$ extends to a map from $\Lp\infty$ to $\Lp 1$ in a straight-forward manner. For general $A\in\Lp\infty$, the image $\nop_\delta[A]\in\Lp 1$ is still positive but might not be normalized.
  From the proof of Proposition \ref{th:quantmax} below, it is easily seen that $\nop_\delta$ maps bijectively onto the positive cone in $\Lp 1$.
\end{remark}
\begin{proof}[Proof of Proposition \ref{th:quantmax}]
  The set
  \[ X := \{ R \in \dens\;| \; R_{kk}= \delta\,n(k\delta) \  \text{for all $k=1,\ldots,N$} \} \]
  is convex and closed. $X$ is not empty since it contains the positive definite diagonal matrix $S^n$ with $S^n_{jj}=\delta\,n(\delta j)$. Further, $X$ is bounded, since semi-definiteness of $R$ implies that $|R_{k\ell}|\le (R_{kk}+R_{\ell\ell})/2\le\max n$. Since the function $r\mapsto r\log r$ is strictly convex on $\setRnn$, so is the functional $\eng$ on $\dens$ by trace convexity, see Lemma \ref{lem:carlen} from the Appendix.  Thus $\eng$ possesses a unique minimum $R^+\in X$, proving unique solvability of \eqref{eq:minprob}. For the rest of the proof, we shall derive the Euler-Lagrange equations for $R^+$ and show that they imply the representation \eqref{eq:qM}.

  First, $R^+$ must be real, since otherwise, $R^+$ and $\overline{R^+}$ --- which have the same eigenvalues and thus have the same entropy value --- would be two different minimizers.
  Second, $R^+$ must be positive definite. Indeed, assume that $R^+ \in \dens$ is optimal in \eqref{eq:minprob} but has a zero eigenvalue. With $S^n\in\densp$ introduced in the beginning of the proof, consider $R^\varepsilon := (1-\varepsilon)R^+ + \varepsilon S$. It is easily seen that $R^\varepsilon\in\densp$ for each $\varepsilon\in[0,1]$, and satisfies the constraint, i.e., $R^\varepsilon_{kk}=\delta\,n(k\delta)$ for all $n=1,\ldots,N$. The eigenvalues of $R^\varepsilon$ vary continuously with $\varepsilon$. Since $R^\varepsilon$ is positive definite for any $\varepsilon>0$ --- recall that $n$ is a positive function --- all eigenvalues are positive for $\varepsilon>0$, but at least one eigenvalue converges to zero as $\varepsilon\searrow0$. Moreover, the function $(0,1)\ni\lambda\mapsto \lambda\log\lambda$ has infinite negative slope at $\lambda=0+$, hence it follows that $\eng[R^\varepsilon]$ is increasing as $\varepsilon\searrow0$. Thus $\eng[R^+]=\lim_{\varepsilon\searrow0}\eng[R^\varepsilon]$ cannot be the minimal value of $\eng$.
    			
    Thus $R^+$ is positive definite and in particular lies in the interior of $\dens$. We may thus consider perturbations $R^\varepsilon=R^++\varepsilon Z$ for self-adjoint $Z\in\setC^{N\times N}$ with vanishing diagonal, assuming $|\varepsilon|$ is sufficiently small. By optimality of $R^+$, we conclude that
    \begin{align}
      \label{eq:EulerLagrange}
      0= \frac{\dd}{\dn\varepsilon}\bigg|_{\varepsilon= 0}\eng[R^++\varepsilon Z]  
      = \tr{(\log R) Z + Z - \hbar^2\Delta_\delta Z} = \tr{[\log R-\hbar^2\Delta_\delta]Z}.
    \end{align}
    For given indices $k\neq\ell$, let $Z$ be the matrix that has zero entries except for a ones at the positions $k\ell$ and $\ell k$.
    The optimality condition \eqref{eq:EulerLagrange} then implies that
    \begin{align*}
      0 = [\log R^+- \hbar^2\Delta_\delta]_{k\ell}.
    \end{align*}
    Since this holds for $k,\ell=1,\ldots,N$ with $k\neq\ell$, we have that $\log R^+= \hbar^2\Delta_\delta$ \emph{except} on the diagonal. That is, there is a diagonal matrix $D\in\setR^{N\times N}$ such that
    \begin{align*}
      \log R^+ = \hbar^2\Delta_\delta + D.
    \end{align*}
    Defining a function $A\in\Lp 2$ by $A(j\delta)=D_{jj}$, we arrive at the representation \eqref{eq:qM}.

    Since the positive definite self-adjoint real matrix $\mm[n]\in\denspr$ possesses a unique self-adjoint real logarithm $L\in\setR^{N\times N}$, i.e., a unique such solution to $\mm[n]=\exp(R)$, the map $\mm$ is actually injective.
  \end{proof}
  \begin{lemma}
    $\eng$ attains its minimal value on $\dens$ at the discrete quantum Maxwellian for the uniform density $\bar n\equiv 1$, that is
    \begin{align}
      \label{eq:Z}
      \mm_\delta[\bar n]=Z_{\hbar^2,\delta}^{-1}\exp(\hbar^2\Delta_\delta),
      \quad\text{with} \quad
      Z_{\hbar^2,\delta} = \sum_{k=1}^{N}e^{-\hbar^2\omega_k},
    \end{align}
    where $\omega_1,\ldots,\omega_N$ are the eigenvalues of the discrete Laplacian $-\Delta_\delta$, see \eqref{eq:DeltaEV}.
    It follows in particular that, uniformly in $\delta$,
    \begin{align}
      \label{eq:dentbelow}
      \min_{R\in\dens}\eng[R] \ge \underline{\dent}_{\hbar^2}:=-\frac{\sqrt\pi}{4\hbar}.
  \end{align}
\end{lemma}
\begin{proof}
  The proof is analogous to that of Proposition \ref{th:quantmax} above, with the only difference that the constraint simplifies to $\tr R=1$: by compactness of $\dens$, there exists a minimizer $R^+$, which is unique by strict convexity of $\eng$. Uniqueness further implies that $R^+$ is real, and by considering convex combinations of $R^+$ with the unit matrix $\one_N$, one concludes that $R^+$ is positive definite. In the Euler-Lagrange equation \eqref{eq:EulerLagrange} for $R^+$, all self-adjoint variations $Z$ are admissible that have vanishing trace (not necessarily vanishing diagonal, as before). It follows that
  \begin{align*}
    R^+=\exp(\hbar^2\Delta_\delta-a\one_N) = e^{-a}\exp(\hbar^2\Delta_\delta),
  \end{align*}
  with a suitable real Lagrange multiplier $a\in\setR$. The latter is uniquely determined by $\tr{R^+}=1$, i.e.,
  \begin{align*}
    e^a = \tr{\exp(\hbar^2\Delta_\delta)} = \sum_{k=1}^Ne^{-\hbar^2\omega_k}.
  \end{align*}
  This provides the representation \eqref{eq:Z}, with $Z_{\hbar^2,\delta}=e^a$.
  By symmetry, $R^+$ has identical entries on the diagonal, namely $\delta$, because of $\tr{R^+}=1$. Thus, the associated density is $\bar n\equiv1$.
  
  For the minimal value of $\eng$, we find accordingly
  \begin{align*}
    \eng[R^+]
    = \tr{R^+\log R^+-\hbar^2\Delta_\delta R^+}
    = \tr{R^+(\hbar^2\Delta_\delta-a\one_N)-\hbar^2\Delta_\delta R^+}
    = -a\tr{R^+} = -a.
  \end{align*}
  By the estimate \eqref{eq:LT} from Appendix \ref{app:Laplace}, we have that
  \begin{align*}
    a = \log\left(\sum_{k=1}^Ne^{-\hbar^2\omega_k}\right) \le \log\left(1+\frac{\sqrt\pi}{4\hbar}\right) \le \frac{\sqrt\pi}{4\hbar},
  \end{align*}
  uniformly in $\delta$, which proves the lower bound \eqref{eq:dentbelow}.
\end{proof}
\begin{corollary}
  For each $R\in\dens$,
  \begin{align}
    \label{eq:H1byH}
    -\tr{\Delta_\delta R} \le \frac2{\hbar^2} \big(\eng[R]-\underline\dent_{\hbar^2/2}\big).
  \end{align}
\end{corollary}
\begin{proof}
  Due to the lower bound \eqref{eq:dentbelow}, we have
  \begin{align*}
    -\frac{\hbar^2}2\tr{\Delta_\delta R}
    = \tr{R\log R-\hbar^2\Delta_\delta R} - \tr{R\log R-\frac{\hbar^2}2\Delta_\delta R}
    \le\eng[R] - \underline\dent_{\hbar^2/2}.
  \end{align*}
  Division by $\hbar^2/2$ yields \eqref{eq:H1byH}.
\end{proof}
\begin{lemma}
  For any $n=\nop_\delta[A]\in\dprbp$, we have that
  \begin{align}
    \label{eq:Aupdown}
    \min A \le -\log Z_{\hbar^2,\delta} \le \max A,
  \end{align}
  with $Z_{\hbar^2,\delta}$ defined in \eqref{eq:Z}.
\end{lemma}
\begin{proof}
  With $\underline a:=\min A$ and $\overline a:=\max A$,
  trace monotonicity, see Lemma \ref{lem:tracemono} in the appendix, implies that
  \begin{align*}
    \tr{\exp(\underline a\one_N+\hbar^2\Delta_\delta)} \le \tr{\exp(A+\hbar^2\Delta_\delta)} \le \tr{\exp(\overline a\one_N+\hbar^2\Delta_\delta)},
  \end{align*}
  and so, since $n\in\dprbp$,
  \begin{align*}
    e^{\underline a} Z_{\hbar^2,\delta} \le 1 \le e^{\overline a} Z_{\hbar^2,\delta}.
  \end{align*}
  Take the logarithm to arrive at \eqref{eq:Aupdown}.
\end{proof}
In \eqref{eq:dent0}, the discrete entropy functional $\dent$ has been defined at hoc. The relation to the relative von Neumann entropy $\eng$ is the following.
\begin{lemma}
  For any $n=\nop_\delta[A]\in\dprbp$, we have that
  \begin{align}
    \label{eq:nA}
    \dent(n) := \eng\big[\mm_\delta[n]\big] = \delta\sum_\xi nA.
  \end{align}
\end{lemma}
\begin{proof}
  Substitute $\mm_\delta[n]=\exp(\hbar^2\Delta_\delta+A)$ into the definition of the entropy, and use cyclicity of the trace:
  \begin{align*}
    \eng\big[\mm_\delta[n]\big]
    &= \tr{\mm_\delta[n]\log\mm_\delta[n]-\hbar^2\Delta\mm_\delta[n]} \\
    &= \tr{\hbar^2\mm_\delta\Delta_\delta + \mm_\delta[n]A - \hbar^2\Delta\mm_\delta[n]} 
    = \tr{\mm_\delta[n]A}
    = \sum_\xi \delta\, n\,A,
  \end{align*}
  which is \eqref{eq:nA}.
\end{proof}
\begin{lemma}
  For $A$ such that $\nop_\delta[A]=n\in\dprbp$, we have that
  \begin{align}
    \label{eq:Abelow}
    A(\xi) \le 2(\hbar/\delta)^2.
  \end{align}
\end{lemma}
\begin{proof}
  Since the largest eigenvalue $\rho$ of $\mm_\delta[n]=\exp(\hbar^2\Delta_\delta+A)\in\densp$ is less than one, the largest eigenvalue $\log\rho$ of $\hbar^2\Delta_\delta+A$ is negative. This implies for any $j\in\{1,\ldots,N\}$ that
  \begin{align*}
    0 > \cmp{\hbar^2\Delta_\delta+A}{jj} = -\frac{2\hbar^2}{\delta^2} + A(j\delta),
  \end{align*}
  which implies \eqref{eq:Abelow}.
\end{proof}
The following is important for solution of the discrete quantum Liouville equation below.
\begin{lemma}
  \label{lem:mmisdifferentiable}
  The map $\mm_\delta$ from $\dprbp$ to $\setR^{N\times N}$ is differentiable.
\end{lemma}
\begin{proof}
  Since images of $\mm_\delta$ have the representation $\mm_\delta[n]=\exp(\hbar^2\Delta_\delta+A)$ with a unique Lagrange multiplier $A$, and since the matrix exponential map is clearly differentiable, it suffices to prove differentiability of the map from $n$ to $A$. This will be done by means of the implicit function theorem below.

  Define the map $\phi:\setR^N\to\setR^N$ from the Lagrange multiplier $A$, here considered as a real $N$-vector, to the real $N$-vector of diagonal elements of the matrix
  \begin{align*}
    f(A) := \exp(\hbar^2\Delta_\delta+A).
  \end{align*}
  Fix some $A_0\in\setR^N$, and consider a correponding spectral decomposition
  \begin{align*}
    \hbar^2\Delta_\delta + A_0 = \sum_{k=1}^N \lambda_kv_kv_k^*
  \end{align*}
  with an orthonormal basis of real eigenvectors $v_k\in\setR^N$. Now let $Q\in\setR^{N}$, $s\in\setR$. By differentiation of the series representation
  \begin{align*}
    f(A_0+sQ)
    = \sum_{n=0}^\infty \frac1{n!}(\hbar^2\Delta_\delta+A_0+sQ)^n
  \end{align*}
  with respect to $s$ at $s=0$, we obtain
  \begin{align*}
    Df(A_0)[Q]
    = \sum_{n=1}^\infty \frac1{n!}\sum_{p=1}^n(\hbar^2\Delta_\delta+A_0)^{p-1}Q(\hbar^2\Delta_\delta+A_0)^{n-p}
    = \sum_{n=1}^\infty \frac1{n!}\sum_{k,\ell=1}^N\sum_{p=1}^n\lambda_k^{p-1}\lambda_\ell^{n-p}v_kv_k^*Qv_\ell v_\ell^*.
  \end{align*}
  Invertibility of $D\phi(A_0)$ follows if can show that the diagonal of $Df(A_0)[Q]$ vanishes only for $Q=0$.
  To prove this, multiply the derivative by $Q=Q^*$ and take the trace,
  \begin{align*}
    \tr{Q\,Df(A)[Q]}
    &= \sum_{n=1}^\infty \frac1{n!}\sum_{k,\ell=1}^N\sum_{p=1}^n\lambda_k^{p-1}\lambda_\ell^{n-p}\tr{Q^*v_kv_k^*Qv_\ell v_\ell^*} \\
    &= \sum_{k,\ell=1}^N \left[\sum_{n=1}^\infty\sum_{p=1}^n\frac{\lambda_k^{p-1}\lambda_\ell^{n-p}}{n!}\right] |v_k^*Qv_\ell|^2.
  \end{align*}
  The expression in the square bracket above is easily seen to be equal to
  \begin{align*}
    \sum_{n=1}^\infty \frac{\lambda_\ell^{n-1}}{n!} \sum_{p=1}^n\left(\frac{\lambda_k}{\lambda_\ell}\right)^{p-1} 
    = \sum_{n=1}^\infty \frac{\lambda_\ell^n}{n!} \frac{(\lambda_k/\lambda_\ell)^n-1}{\lambda_k/\lambda_\ell-1} 
    = \frac1{\lambda_k-\lambda_\ell}\sum_{n=1}^\infty \left(\frac{\lambda_k^n}{n!}-\frac{\lambda_\ell^n}{n!}\right)
    = \frac{e^{\lambda_k}-e^{\lambda_\ell}}{\lambda_k-\lambda_\ell},
  \end{align*}
  the logarithmic mean of $e^{\lambda_k}$ and $e^{\lambda_\ell}$, and is in particular positive for any values of $\lambda_k$ and $\lambda_\ell$. Thus,
  \begin{align*}
    \tr{Q\,Df(A)[Q]}
    \ge \sum_{k=1}^N e^{\lambda_k}|v_k^*Qv_k|^2,
  \end{align*}
  which is zero if and only if $Q=0$.
\end{proof}
Finally, we introduce the next-to-diagonal entries in the quantum Maxwellian by
\begin{align}
  \label{eq:nu}
  \nu^+(k\delta) = \delta^{-1}\cmp{\mm_\delta[n]}{k,k+1},
  \quad
  \nu^-(k\delta) = \delta^{-1}\cmp{\mm_\delta[n]}{k,k-1}.
\end{align}
\begin{lemma}
  \label{lem:nu}
  For any $n\in\dprbp$, the coefficients $\nu^+$ and $\nu^-$ are positive, and
  \begin{align}
    \label{eq:nusum}
    \delta\sum_\xi \nu^+ \le1,
    \qquad
    \delta\sum_\xi \nu^- \le1.
  \end{align}
\end{lemma}
\begin{proof}
  The proof of positivity rests on the fact that for a real matrix $M\in\setR^{N\times N}$ with non-negative entries, its matrix exponential $\exp(M)$ also has non-negative entries, and if additionally $M_{k\ell}>0$ for some $k,\ell\in\{1,\ldots,N\}$, then $\cmp{\exp(M)}{k\ell}\ge M_{k\ell}>0$. This follows immediately from the series representation, specifically from
  \begin{align*}
    \cmp{\exp(M)}{k\ell} = \sum_{n=0}^\infty \frac{\cmp{M^n}{k\ell}}{n!} \ge \frac{\cmp{M^1}{k\ell}}{1!} = M_{k\ell}.
  \end{align*}
  Fix $n\in\dprbp$, and let $a:=\min_\xi A(\xi)$. By definition, the matrix $\Delta_\delta+2\delta^{-2}\one_N$ is zero except for the constants $\delta^{-2}$ on the off-diagonals, and has in particular only non-negative entries. Thus,
  \begin{align*}
    \delta \nu^+(k\delta)
    = e^{a-2\hbar^2/\delta^2}\big[\exp\big((A-a\one_N) + \hbar^2(\Delta_\delta+2\delta^{-2}\one_N)\big)\big]_{k,k+1}
    \geq e^{a-2(\hbar/\delta)^2} (\hbar/\delta)^2 > 0.
  \end{align*}
  For proving the bound \eqref{eq:nusum}, it suffices to observe that, since $R=\mm_\delta[n]$ is real and positive definite, $|R_{j,j+1}|\le (R_{j,j}+R_{j+1,j+1})/2$, and consequently that
  \begin{align*}
    \nu^+(\xi) \le \frac12\big(n(\xi)+n(\xi+\delta)\big).
  \end{align*}
  Now sum over $\xi\in\crcN$ and use that $n\in\dprbp$ is normalized to one.
\end{proof}
%

\subsection{The discrete quantum Liouville equation}
We introduce the discrete collisional quantum Liouville equation as follows:
\begin{equation}
  \label{eq:DMQL} 
  i\hbar \dot R = \left[-\frac{\hbar^2}2\Delta_\delta, R\right] + \frac{i}\tau\left( \mm_\delta[R] - R \right),
\end{equation}
where $\mm_\delta[R]$ is the projection of $R\in\dens$ onto the manifold of discrete quantum Maxwellians, see Definition \ref{dfn:nop}, and the relaxation time $\tau>0$ is a positive parameter.
%
\begin{lemma}
  For any initial condition $R_\text{in}\in\densp$, there exists a unique global solution $R\in C^1(\setRnn;\densp)$ to \eqref{eq:DMQL}.
\end{lemma}
\begin{proof}
  By Lemma \ref{lem:mmisdifferentiable}, the right-hand side of the ODE \eqref{eq:DMQL} is a differentiable function w.r.t.\ $R$ on $\densp$. Moreover, the right-hand side is tangent to $\densp$, i.e., it is self-adjoint of trace zero for any argument $R\in\densp$. By the existence and uniqueness theorem for ODEs, the initial value problem possesses a unique maximal local solution $R:[0,T)\to\densp$. To conclude that actually $T=+\infty$, i.e., that $R$ is global, it suffices to assume $T\in(0,\infty)$, and prove that $R(t)$ stays away from the boundaries of $\densp$, uniformly on $[0,T)$. The boundary of the manifold $\densp$ in the space $\setC^{N\times N}$ is given by the unit-trace  self-adjoint matrices that are positive semi-definite but \emph{not positive definite}.

  To verify the preservation of positive definiteness, we adapt an idea from \cite{DeMeRi05}: consider the auxiliary curve $S:[0,T)\to\setC^{N\times N}$ given by
  \begin{align*}
    S(t) = \exp\left(-\frac{i\hbar}2 t\Delta_\delta\right) e^{t/(\hbar\tau)} R(t)\exp\left(\frac{i\hbar}2 t\Delta_\delta\right).
  \end{align*}
  Notice that $S(0)=R(0)$ is positive definite.
  Using that $\frac{d}{dt}\exp(tA)=A\exp(tA)=\exp(tA)A$ for any square matrix $A$ and real $t$, we conclude that $S$'s time derivative is given by
  \begin{align*}
    \hbar\dot S(t)
    &= - i\frac{\hbar^2} 2[\Delta_\delta,S(t)] + \frac1\tau S(t) + \exp\left(-\frac{i\hbar}2 t\Delta_\delta\right) e^{t/(\hbar\tau)}\hbar\dot R(t)\exp\left(\frac{i\hbar}2 t\Delta_\delta\right) \\
    &= \exp\left(-\frac{i\hbar}2 t\Delta_\delta\right) e^{t/(\hbar\tau)}\left(\frac{\mm_\delta[R(t)]-R(t)}{\tau}+\frac{R(t)}{\tau}\right)\exp\left(\frac{i\hbar}2 t\Delta_\delta\right) \\
    &= \frac1\tau\exp\left(-\frac{i\hbar}2 t\Delta_\delta\right) e^{t/(\hbar\tau)}\mm_\delta[R(t)]\exp\left(\frac{i\hbar}2 t\Delta_\delta\right).                
  \end{align*}
  Since $\exp(\pm (i\hbar t)/2\,\Delta_\delta)$ are unitary matrices, and since quantum Maxwellians are positive semi-definite, the right-hand side above is positive semi-definite as well. Hence $S(t)\ge S(0)=R(0)$ for all times $t\in[0,T)$. By construction of $S$, this implies the desired positive lower bound $R(t)\ge e^{-t/(\hbar\tau)}R(0) \ge e^{-T/(\hbar\tau)}R(0)$,  uniformly on $[0,T)$.
\end{proof}
  
\subsection{Diffusive limit}
To study the diffusive limit of the dynamics given by \eqref{eq:DMQL}, we assume that the relaxation time $\tau:=\eps/\hbar>0$ is small, and we consider the rescaled solutions $R^\eps(t,x)=R\big(2t/\eps,x\big)$, i.e., we study large times. Substitute this into \eqref{eq:DMQL} and divide by $i\hbar/2$ to obtain:
\begin{equation}
  \label{eq:DMQLeps} 
  \eps \dot R^\eps = i\hbar\commu{\Delta_\delta}{R^\eps} + \frac{2}{\eps}(\mm_\delta[R^\eps] - R^\eps).
\end{equation}
\begin{proposition}[Diffusive Limit]\label{prop:DQD}
  For a given initial datum $R^\text{in} \in \densp$, consider the family $(R^\eps)_{\eps\in(0,1)}$ of corresponding solutions $R^\eps:[0,\infty)\to\densp$ to \eqref{eq:DMQLeps}. Then there exist a differentiable curve $R^0:[0,\infty)\to\densp$ in the manifold of discrete quantum Maxwellians, i.e., $R^0(t)=\mm_\delta[R^0(t)]$ for all $t\ge 0$, such that, along a suitable sequence $\eps\searrow0$, we have that
  \begin{enumerate}
  \item $R^\eps\wto R^0$ in $L^2_\text{loc}([0,\infty);\setC^{N\times N})$,
  \item $R^\eps_{kk}(t)\to R^0_{kk}(t)$ for $k=1,\ldots,N$, locally uniformly with respect to $t\ge0$.
  \end{enumerate}
  Finally, the diagonal entries $R^0_{kk}$ satisfy the ODEs
  \begin{align}
    \label{eq:preDQDD}
    \dot R^0_{kk} = -\frac{\hbar^2}2\,\commu{\Delta_\delta}{\commu{\Delta_\delta}{R^0}}
  \end{align}
  for all $k=1,\ldots,N$, with respective initial conditions $R^0_{kk}(0)=R^\text{in}_{kk}$.
\end{proposition}
\begin{remark}
  Two comments on the limiting initial value problem:
  \begin{enumerate}
  \item Because of $R^0=\mm_\delta[R^0]$, equation \eqref{eq:preDQDD} is a actually a closed ODE system for the $N$ diagonal entries of $R^0$; the off-diagonal elements are fully enslaved. The respective ODE system is given more explicitly below in \eqref{eq:DQDD}.
  \item By locally uniform convergence, the diagonal elements $R^\eps_{kk}(t)$ converge to $R^\text{in}_{kk}$ as $\eps\searrow0$ and $t\searrow$, for each $k=1,\ldots,N$. On the other hand, the off-diagonal elements $R^0(t)_{k\ell}$ are implicitly given by the relation $R^0(t)=\mm_\delta[R^0(t)]$, so even if the limit $R^\eps_{k\ell}(t)$ for $\eps\searrow0$ and $t\searrow0$ exists, it will not be consistent with $R^\text{in}_{k\ell}$ in general.  
 \end{enumerate}
\end{remark}
\begin{proof}
  Fix some terminal time $T>0$. We consider $R^\eps$ as a $C^1$-curve  in $\setC^{N\times N}$ over $[0,T]$. Since all $R^\eps(t)$ are density matrices and thus have unit Hilbert-Schmidt norm, there is a sequence $\eps\searrow0$ such that $R^\eps\wto R^0$ in $L^2([0,T];\setC^{N\times N})$. And since $\dens$ is a closed and convex subset of $\setC^{N\times N}$, we have that $R^0(t)\in\dens$ for a.e. $t\in[0,T]$; note that we do not yet know that $R^0(t)\in\densp$. The goal of the following is to identify the weak limit $R^0$.

  We begin by proving the a priori estimate
  \begin{align}
    \label{eq:energyest}
    \eng[R^\eps(T)] + 2\int_0^T \left\|\frac{\mm_\delta[R^\eps(t)]-R^\eps(t)}\eps\right\|^2\dd t \le \eng[R^\eps(0)].
  \end{align}
  For the derivative of the free energy, we find
  \begin{align*}
    -\eps\frac{\dd}{\dn t} \eng[R^\eps]
    &= -\tr{\eps\dot R^\eps(\log R^\eps-\hbar^2\Delta_\delta)}\\
    &= i\hbar \tr{\commu{\Delta_\delta}{R^\eps} (\hbar^2\Delta_\delta-\log R^\eps)}
      +\frac2\eps\tr{\big(\mm_\delta[R^\eps] - R^\eps\big) (\hbar^2\Delta_\delta-\log R^\eps)}.
  \end{align*}
  Using cyclicity of the trace, and that $R^\eps$ commutes with $\log R^\eps$, it follows that the first term in the last line above vanishes.
  To simplify the second term, observe that $\log\mm_\delta[R^\eps]=\hbar^2\Delta_\delta+A$ for a suitable Lagrange multiplier $A$. This allows to replace $\hbar^2\Delta_\delta$ by $\log\mm_\delta[R^\eps]-A$ above. Further observe that the product of the diagonal matrix $A$ with the difference $\mm_\delta[R^\eps]-R^\eps$, which has vanishing diagonal, does not contribute to the trace. Thus, after division by $\eps>0$:
  \begin{align*}
    - \frac{\dd}{\dn t} \eng[R^\eps]
    =\frac2{\eps^2}\tr{(\mm_\delta[R^\eps] - R^\eps) (\log\mm_\delta[R^\eps]-\log R^\eps)}
    \ge 2\left\|\frac{\mm_\delta[R^\eps] - R^\eps}\eps\right\|^2,
  \end{align*}
  where the last estimate above is a consequence of the non-commutative functional inequality from Corollary \ref{lem:kleincor}.
  Integration in time yields \eqref{eq:energyest}.

  There are two immediate consequences of \eqref{eq:energyest}: first, passing to a subsequence if necessary, there is a $Z\in L^2([0,T];\setC^{N\times N})$ such that
  \begin{align}
    \label{eq:identifyZ}
    \frac{\mm_\delta[R^\eps] - R^\eps}\eps \wto Z \qquad \text{in $L^2([0,T];\setC^{N\times N})$;}
  \end{align}
  we shall identify $Z$ further below.
  And second, we have that
  \begin{align}
    \label{eq:identifyMM}
    \mm_\delta[R^\eps]\wto R^0 \qquad \text{in $L^2([0,T];\setC^{N\times N})$},
  \end{align}
  where $R^0$ is the weak limit of $R^\eps$. Below, shall use this fact to show that $R^0(t)=\mm_\delta[R^0(t)]$ for a.e. $t\in[0,T]$, i.e., that $R^0$ is a curve in the manifold of discrete quantum Maxwellians. However, since that manifold is not a convex subset of $\setC^{N\times N}$, the weak convergence of $\mm_\delta[R^\eps]$ is not yet enough to draw that conclusion.
  
  With these goals in mind, we study the evolution of $R^\eps$'s diagonal elements under \eqref{eq:DMQLeps}. Since $\mm_\delta[R^\eps]_{kk}=R^\eps_{kk}$ by definition of the discrete quantum Maxewellian, the relaxation term vanishes on the diagonal. So \eqref{eq:DMQLeps} simplifies to
  \begin{align*}
    \dot R^\eps_{kk} = \frac{i\hbar}\eps\,\commu{\Delta_\delta}{R^\eps}_{kk}
    = -i\hbar\,\commu{\Delta_\delta}{\frac{\mm_\delta[R^\eps]-R^\eps}\eps}_{kk},
  \end{align*}
  where the second inequality above follows since both $\Delta_\delta$ and  $\mm_\delta[R^\eps]$ are real matrices, hence their commutator has vanishing diagonal. By \eqref{eq:energyest}, it thus follows that the diagonal elements $R^\eps_{kk}$ converge weakly to their respective limits $R^0_{kk}$ not only in $L^2([0,T])$ but even in $H^1([0,T])$, and that
  \begin{align}
    \label{eq:identifyEQ}
    \dot R^0_{kk} = -i\hbar\,\commu{\Delta_\delta}{Z}_{kk}.
  \end{align}
  Weak convergence in $H^1$ implies uniform convergence of $R^\eps_{kk}$ to $R^0_{kk}$, and that the latter is continuous. Since $R^\eps_{kk}(0)=R^\text{in}_{kk}$ are all positive, one can choose $T>0$ such that the functions $R^0_{kk}$ are positive on $[0,T]$; we shall prove later that $T>0$ can be chosen arbitrarily. By the differentiability of the map $R\mapsto\mm_\delta[R]$, see Lemma \ref{lem:mmisdifferentiable} above, we conclude that also $\mm_\delta[R^\eps]$ converges uniformly on $[0,T]$ to the respective limit $\mm_\delta[R^0]$. Recalling \eqref{eq:identifyMM} above, we conclude that $R^0(t)=\mm_\delta[R^0(t)]$ is a Maxwellian at a.e. $t\in[0,T]$.

  It remains to identify the weak limit $Z$ from \eqref{eq:identifyZ}, then \eqref{eq:identifyEQ} yields the desired evolution equation. Observe that the $H^{-1}$-norm of the distributional time derivative of a bounded function $f:[0,T]\to\setR$ is controlled in terms of $|f(T)-f(0)|$. Specifically, integrate \eqref{eq:DMQLeps} on some interval $[t_1,t_2]\subseteq[0,T]$ to obtain
  \begin{align*}
    \eps\big(R^\eps(t_2)-R^\eps(t_1)\big) = i\hbar\int_{t_1}^{t_2}\commu{\Delta_\delta}{R^\eps}\dd t + 2\int_{t_1}^{t_2}\frac{\mm_\delta[R^\eps]-R^\eps}\eps\dd t.
  \end{align*}
  In the limit $\eps\searrow0$, the left-hand side converges to zero, and we are left with
  \begin{align*}
    0 = i\hbar\int_{t_1}^{t_2}\commu{\Delta_\delta}{R^0}\dd t + 2\int_{t_1}^{t_2}Z\dd t.
  \end{align*}
  Since this holds for arbitrary $[t_1,t_2]\subseteq[0,T]$, it follows that
  \begin{align*}
    Z = -\frac {i\hbar}2\commu{\Delta_\delta}{R^0} = -\frac {i\hbar}2\commu{\Delta_\delta}{\mm_\delta[R^0]}.
  \end{align*}
  Insert this into \eqref{eq:identifyEQ} to arrive at \eqref{eq:preDQDD}.

  Having established all of the above, it is now easy to conclude that the diagonal elements of $R^0$ are actually a classical solution to \eqref{eq:preDQDD} and in particular are differentiable. To show that the terminal time $T>0$ can be chosen arbitrarily, recall that the only restriction on $T$ has been that the limits $R^0_{kk}$ of the diagonals are positive on $[0,T]$.
  In Lemma \ref{lem:thisisDQDD}, we prove that the $R^0_{kk}$ satisfy the discrete nlQDD system \eqref{eq:DQDD}, and in Proposition \ref{prp:positivity} further below, we show that the unique solution to that ODE system is positive and global. By uniform convergences of the diagonal elements of $R^\eps$ to those of $R^0$ on $[0,T]$, this suffices to conclude that $T>0$ is indeed arbitrary.
\end{proof}
\begin{lemma}
  \label{lem:thisisDQDD}
  Consider a differentiable curve $R^0:[0,T]\to\denspr$ in the manifold of discrete quantum Maxwellians,
  \begin{align*}
    R^0(t)=\mm_\delta[R^0(t)] = \exp\big(\hbar^2\Delta_\delta+A(t)\big)
  \end{align*}
  where $A(t)$ is the Lagrange multiplier for the Maxwellian at time $t$. And let $n:[0,T]\to\dprbp$ be the associated time-dependent density, i.e., $\delta\, n(t;k\delta)=R^0_{kk}(t)$.
  Then $R^0$ satisfies the ODE system \eqref{eq:preDQDD} if and only if $n$ is a solution to the discrete QDD equation
  \begin{align}
    \label{eq:DQDD}
    \dot n = \dq_\delta^-(\nu^+\dq_\delta^+A),
  \end{align}
  where, according to \eqref{eq:nu}, $\delta\,\nu^+(t;k\delta)=\big\{R^0(t)\big\}_{k+1,k}$.
\end{lemma}
\begin{proof}
  Rewrite $\Delta_\delta$ as follows:
 \begin{align*}
   \Delta_\delta = \delta^{-2}\big(L+L^T - 2\,\one_N\big),
   \quad\text{with}\quad
   L = \begin{pmatrix}
     0& 1 &  &  &  \\
     &  0& 1 &  &  \\
     &   &  \ddots&  \ddots &  \\
     &  &  &  0& 1 \\
     1 & & &  &0 
   \end{pmatrix}.
 \end{align*}
 $L$ has the property that, for any $M\in\setC^{N\times N}$,
 \begin{align*}
   \{LM\}_{k,\ell} = M_{k+1,\ell}
   \quad\text{and}\quad
   \{ML\}_{k,\ell} = M_{k,\ell-1},
 \end{align*}
 where addition and subtraction of indices is again considered cyclically on $\{1,\ldots,N\}$.
 Now if $S\in\setR^{N\times N}$ is some real anti-symmetric matrix, then
 \begin{align*}
   \commu{\Delta_\delta}{S}_{kk}
   = \delta^{-2}\big(\commu LS_{kk}+\commu{L^T}{S}_{kk}\big)
   &= \delta^{-2}\big(\commu LS_{kk} + \commu {S^T}{L}^T_{kk}\big) \\
   &= 2\delta^{-2}[LS-SL]_{kk}
   = 2\delta^{-2}(S_{k+1,k}-S_{k,k-1}).
 \end{align*}
 Before we substitute $S:=-\hbar^2\commu{\Delta_\delta}{R^0}$ above, we perform some simplifications of $S$. Since $R^0=\exp(\hbar^2\Delta_\delta+A)$ and thus $\commu{\hbar^2\Delta_\delta+A}{R^0}=0$, we obtain
 \begin{align*}
   -\hbar^2\commu{\Delta_\delta}{R^0}
   = -\commu{(\hbar^2\Delta_\delta+A)-A}{R^0}
   = \commu A{R^0}.
 \end{align*}
 Since $A$ is a diagonal matrix, we have that $\{AR^0\}_{k\ell}=A_kR^0_{k\ell}$ and $\{R^0A\}_{k\ell}=A_\ell R^0_{k\ell}$, and therefore,
 \begin{align*}
   -\hbar^2\commu{\Delta_\delta}{R^0}_{k+1,k} = R^0_{k+1,k}\, (A_{k+1}-A_k),
   \quad
   -\hbar^2\commu{\Delta_\delta}{R^0}_{k,k-1} = R^0_{k,k-1}\, (A_{k}-A_{k-1}).
 \end{align*}
 Substitute these expressions above and recall the definition \eqref{eq:nu} of the $\nu$'s: we have shown that
 \begin{align*}
   -\frac{\hbar^2}{2}\commu{\Delta_\delta}{\commu{\Delta_\delta}{R^0}}_{kk} = \delta\, \big[\dq^-_\delta\big(\nu^+\dq_\delta^+A)\big]_k.
 \end{align*}
 To see the equivalence between \eqref{eq:preDQDD} for $R^0$ and \eqref{eq:DQDD} for $n$, it suffices to observe that the respective initial value problems for the ODE systems are uniquely solvable.
\end{proof}


\section{Properties of the discrete nlQDD}
\label{sct:nlQDD}
In this section, we prove several qualitative properties of the discrete nlQDD system \eqref{eq:DQDD} that have already been summarized in Theorem \ref{thm:2}.

\subsection{Entropy dissipation}
Recall that the discrete entropy functional $\dent$ is given by \eqref{eq:nA}, and the lower bound $\underline{\dent}_{\hbar^2}$ defined in \eqref{eq:dentbelow}.
\begin{proposition}
  \label{prp:dentdiss}
  Let $n:[0,T)\to\dprbp$ be a local solution to the discrete nlQDD \eqref{eq:DQDD}.
  Then $t\mapsto\dent(n(t))$ is non-increasing, and moreover,
  \begin{align}
    \label{eq:dentdiss}
    \int_0^T\delta\sum_\xi\nu^+\big(\dq_\delta^+A\big)^2 \le \dent(n(0))-\underline{\dent}_\hbar.
  \end{align}
\end{proposition}
\begin{proof}
  Substitute $\mm_\delta[n]=\exp(\hbar^2\Delta_\delta+A)$ into $\dent$:
  \begin{align*}
    \frac{\dd}{\dn t}\dent(n)
    &= \frac{\dd}{\dn t}\tr{\mm_\delta[n]\log\mm_\delta[n] -\hbar^2\Delta_\delta\mm_\delta[n] } \\
    &= \tr{(1+\log\mm_\delta[n]-\hbar^2\Delta_\delta)\partial_t\mm_\delta[n]}
      = \tr{(1+A)\partial_t\mm_\delta[n]}.
  \end{align*}
  Since $(1+A)$ is a diagonal matrix, and $\mm_\delta[n]$ has the values $\delta n$ on its diagonal, the trace simplifies as follows:
  \begin{align*}  
    - \frac{\dd}{\dn t}\dent(n)
    &= -\delta\sum_\xi(1+A)\dot n \\
    &= -\delta\sum_\xi(1+A)\,\dq_\delta^-(\nu_\delta^+\,\dq^+A)
    = \delta\sum_\xi\dq_\delta^+(1+A)\,\nu\,\dq_\delta^+A
    = \delta\sum_\xi \nu\,\big(\dq_\delta^+A\big)^2,
  \end{align*}
  where we have used the mutual anti-adjointness of $\dq_\delta^+$ and $\dq_\delta^-$ in the $\Lp 2$ scalar product.
  It now follows that $t\mapsto\dent(n(t))$ is non-increasing.
  Moreover, integration in $t\in[0,T]$ and an application of the lower bound \eqref{eq:dentbelow} yields \eqref{eq:dentdiss}.
\end{proof}
\begin{corollary}
  \label{cor:DQDDhoelder}
  A local solution $n:[0,T)\to\dprbp$ to \eqref{eq:DQDD} is H\"older continuous, uniformly in time.
\end{corollary}
This is only an auxiliary result needed for the proof of Proposition \ref{prp:positivity} below. A $\delta$-uniform H\"older bound is given later in the proof of Proposition \ref{prp:nconv}.
\begin{proof}
  For arbitrary $t_2>t_1\ge0$, and at any $\xi\in\crcN$,
  \begin{align*}
    \big|n(t_2;\xi)-n(t_1;\xi)\big|
    &= \left|\int_{t_1}^{t_2}\dot n(t;\xi)\dd t\right| \\
    & \le \int_{t_1}^{t_2} \big|\dq_\delta^-\big(\nu^+(\dq_\delta^+A)\big)\big|(t;\xi)\dd t \\
    & \le \frac1\delta \int_{t_1}^{t_2} \Big(\nu^+\big|\dq_\delta^+A\big|(t;\xi) + \nu^-\big|\dq_\delta^-A\big|(t;\xi)\Big)\dd t.
  \end{align*}
  Summation over $\xi$ and an estimate by the Cauchy-Schwarz inequality yield
  \begin{align*}
    \|n(t_2)-n(t_1)\|_{\Lp 1}
    &\le \frac2\delta \int_{t_1}^{t_2} \delta\sum_\xi \nu^+\big|\dq_\delta^+A\big|\dd t \\
    & \le \frac2\delta \left(\int_0^T\delta\sum_\xi \nu^+\big(\dq_\delta^+A\big)^2\dd t\right)^{1/2}\left(\int_{t_1}^{t_2}\delta\sum_\xi \nu^+\dd t\right)^{1/2} \\
    &\le \frac2\delta \left(\dent(n(0))-\underline{\dent}_{\hbar^2}\right)^{1/2}|t_2-t_1|^{1/2},
  \end{align*}
  where we have used \eqref{eq:nusum} and \eqref{eq:dentdiss} in the last estimate.
\end{proof}

\subsection{Positivity preservation}
This is section is devoted to proving that solutions to the discrete nlQDD system \eqref{eq:DQDD} preserve positivity. We start with an auxiliary result.
\begin{lemma}
  \label{lem:n2A}
  Consider a sequence $A_k$ for which $n_k=\nop_\delta[A_k]\in\dprbp$ converges to $n_*\in\dprb$. For each $\xi_*\in\crcN$, we have that $n_*(\xi_*)=0$ if only if $A_k(\xi_*)\to-\infty$. 
\end{lemma}
\begin{proof}
  Concerning the \emph{if}:
  recall from the uniform lower bound in \eqref{eq:Abelow} that $A_k(\xi)\le 2(\hbar/\delta)^2$ for all $\xi\in\crcN$. Therefore, and since $n_k\in\dprbp$, the sum of $\delta n_k(\xi)A_k(\xi)$ over any selection of indices $\xi\in\crcN$ is less or equal to $2(\hbar/\delta)^2$ as well. By the lower bound on the entropy from \eqref{eq:dentbelow}, we have for every $\xi_*\in\crcN$ that
  \begin{align*}
    \underline{\dent}_{\hbar^2}\le \eng[n_k] = \delta\sum_\xi n_k(\xi)A_k(\xi) \le n_k(\xi_*)A_k(\xi_*) + 2(\hbar\delta)^2.
  \end{align*}
  And so, in particular, if $A_k(\xi_*)\to-\infty$ as $k\to\infty$, then also $n_k(\xi_*)\to0$.
 
  Concerning the \emph{only if}:
  choose $j\in\{1,\ldots,N\}$ such that $n_*(\xi_*)\to0$ for $\xi_*=j\delta$. Since $\dens$ is a compact subset of $\setC^{N\times N}$, there is a (non-relabeled) subsequence such that $\mm_\delta[n_k]\to \tilde\mm$ for some $\tilde\mm\in\dens$. By hypothesis, we have that
  \[ \tilde\mm_{jj}=\lim_{k\to\infty}\cmp{\mm_\delta[n_k]}{jj}=\lim_{k\to\infty}n_k(\xi_*) = n_*(\xi_*) = 0. \]
  For each $k$, let $\rho_{1,k},\ldots,\rho_{N,k}\in(0,1)$ be the eigenvalues of $\mm_\delta[n_k]$, and $v_{1,k},\ldots,v_{N,k}\in\setR^N$ be an associated orthonormal system of real eigenvectors, i.e., $v_{\ell,k}^Tv_{\ell',k}=0$ for $\ell'\neq\ell$ and $v_{\ell,k}^Tv_{\ell,k}=1$. Recall that orthonormality implies that $\sum_{\ell=1}^N\cmp{v_{\ell,k}}j^2=1$ for every component $j=1,\ldots,N$. Now since
  \begin{align*}
    \sum_{\ell=1}^N\rho_{\ell,k}\big[v_{\ell,k}\big]_j^2 = e_j^T\mm_\delta[n_k]e_j = \delta n_k(\xi_*) \to 0
  \end{align*}
  as $k\to\infty$, it follows by Jensen's inequality that
  \begin{align*}
    -2(\hbar/\delta)^2 + A_k(\xi_*) 
    &= \cmp{\hbar^2\Delta_\delta+A_k}{jj} \\
    & = \cmp{\log\mm_\delta[n_k]}{jj} \\
    & = \sum_{\ell=1}^N\log\rho_{\ell,k}\cmp{v_{\ell,k}}j^2
    \le \log\left[\sum_{\ell=1}^N\rho_{\ell,k}\cmp{v_{\ell,k}}j^2\right]
    = \log\big(\delta n_k(\xi_*)\big) \to -\infty,
  \end{align*}
  which implies in particular that $A_k(\xi_*)\to-\infty$.
\end{proof}
\begin{proposition}
  \label{prp:positivity}
   For any initial datum $n^\text{in}\in\dprbp$, the corresponding solution to \eqref{eq:DQDD} is positive at all times and hence global.
\end{proposition}
\begin{proof}
  For a contradiction assume that there is some solution $n:[0,T)\to\dprbp$ to the discrete QDD that does not remain positive, i.e., $\lim_{t\nearrow T}\min n(t)=0$.
  Note that this is a genuine limit for $t\nearrow T$, not only along a sequence $t_n\nearrow T$, by the uniform H\"older regularity $n$ in time, see Corollary \ref{cor:DQDDhoelder}.
  For the same reason, $n(t)$ possesses a limit $n_*$ for $t\nearrow T$. Since $n_*\in\dprb$ cannot be identically zero, there must exist a set
  \begin{align*}
    \Xi=\{\underline\xi,\underline\xi+\delta,\ldots,\overline\xi-\delta,\overline\xi\}\subset\crcN
  \end{align*}
  of points $\xi\in\Xi$ for which $n(t;\xi)\searrow0$ as $t\nearrow T$, but both $n(t;\overline\xi+\delta)$ and $n(t;\underline\xi-\delta)$ remain bounded away from zero.
  Lemma \ref{lem:n2A} above implies that
  \begin{align}
    \label{eq:A2minus}
    A(t;\xi)\to-\infty\quad\text{for all $\xi\in\Xi$}.
  \end{align}
  Notice that there might be several such sets $\Xi$, and that $\underline\xi-\delta$ and $\overline\xi+\delta$ might represent the same point in $\crcN$. In any case, define
  \begin{align*}
    m(t):=\delta^2\sum_{\xi\in\Xi} n(t;\xi).
  \end{align*}
  By definition of $\Xi$, we have that $m(t)>0$ for $t<T$ and that $m(t)\searrow0$ as $t\nearrow T$. Consequently, there exists an increasing sequence $(t_k)$ of times $t_k\in(0,T)$ that converges to $T$, and which is such that $\dot m(t_k)<0$. Since $n$ is a solution to \eqref{eq:DQDD}, 
  \begin{align*}
    \dot m
    &= \delta^2\sum_{\xi\in\Xi}\dot n(\xi)
    = \sum_{\xi\in\Xi} \big[\nu^+(\xi)\big(A(\xi+\delta)-A(\xi)\big)-\nu^+(\xi-\delta)\big(A(\xi)-A(\xi-\delta)\big)\big] \\
    &= \nu^+(\overline\xi) \big(A(\overline\xi+\delta)-A(\overline\xi)\big)
    + \nu^+(\underline\xi-\delta) \big(A(\underline\xi-\delta)-A(\underline\xi)\big).
  \end{align*}
  Since $\nu^+$ has positive values by Lemma \ref{lem:nu}, the negativity of $\dot m(t_k)$ implies that for each $k$, at least one of the following inequalities is true:
  \begin{align*}
    A(t_k;\overline\xi+\delta)\le A(t_k;\overline\xi),
    \quad\text{or}\quad
    A(t_k;\underline\xi-\delta)\le A(t_k;\underline\xi).    
  \end{align*}
  Without loss of generality, we may assume that the first is true for all $t_k$. Since \eqref{eq:A2minus} holds in particular for $\xi=\overline\xi$, it follows that $A(t_k;\overline\xi+\delta)\to-\infty$. This, in turn, implies by means of Lemma \ref{lem:n2A} that $n(t_k; \overline\xi+\delta)\to0$. It follows that $n_*(\overline\xi+\delta)=0$, contradicting the definition of $\Xi$.
\end{proof}

\section[Quantum exponentials]{The continuous and the discrete quantum exponentials}
\label{sct:hell}
%
We shall now study in more detail the discrete quantum Maxwellians and their continuous limit. Our quantitative comparison between the two is a substantial ingredient for the convergence proof in Section \ref{sct:convergence}.

\subsection{The continuous quantum exponential}
The main goal of the considerations below is to define a map $\nop$ from potentials $A\in L^1(\crc)$ to densities $n\in L^\infty(\crc)$ that can be understood as the limit for $\delta\searrow0$ of the discrete quantum exponential maps $\nop_\delta$ introduced in Proposition \ref{th:quantmax}. With that goal in mind, we use an approach that is complementary to the original definition of $\nop[A]$ in \cite{Pinaud19}: we introduce, seemingly ad hoc, a kernel function as fixed point of a integral operator and then verify a posteriori that the associated Hilbert-Schmidt operator is $\exp(\hbar^2\Delta+A)$. 
This indirect definition gives more direct access to the approximation error of $\nop$ by $\nop_\delta$ in a hands-on way.
\smallskip

In the following, we consider potentials
\begin{align}
  \label{eq:CA}
  A\in \admA:=\big\{ A\in L^1(\crc)  \,\big|\, \|A\|_{L^1}\le C_A \big\},
\end{align}
where $C_A>0$ is fixed.
\begin{definition}
    At each $t>0$, define the \emph{heat kernel} $\htkrnl_t\in C^\infty(\crc)$ by
    \begin{align}
    	\label{eq:dfnhtkrnl}
        \htkrnl^t(z) = \sum_{k\in\setZ} e^{-(2\pi k)^2\hbar^2 t}e^{2\pi ik z}.
    \end{align}
\end{definition}
\begin{lemma}
	$\htkrnl$ is the rescaled one-periodic fundamental solution to the heat equation, that is,
	\begin{align}
          \label{eq:eqhtkrnl}
          \htkrnl^t(z) = \sum_{m\in\setZ} \heat^{\hbar^2 t}(z-m)
          \quad\text{for}\quad
          \heat^\sigma(y) = \frac1{\sqrt{4\pi \sigma}}\exp\left(-\frac{y^2}{4\sigma}\right). 
	\end{align}
	In particular, one has, at every $t>0$,
	\begin{align}
		\label{eq:pde-htkrnl}
		\partial_t\htkrnl^t = \hbar^2 \partial_{zz}\htkrnl^t,
	\end{align}
	and moreover, for every $\varphi\in C(\crc)$, 
	\begin{align}
          \label{eq:ic-htkrnl}
          \int_\crc \varphi(x)\htkrnl^t(x-y)\dd x \to \varphi(y) \quad\text{uniformly in $y\in\crc$ as $t\searrow0$}.          
	\end{align}
	Finally, $\htkrnl^t$ is positive and of unit mass for each $t>0$,
        and there are finite constants $C_\hbar$ and $C_\hbar'$ such that
	\begin{align}
          \label{eq:bd-htkrnl}
          \big\|\htkrnl^t\big\|_{L^\infty(\crc)} \le C_\hbar t^{-1/2} ,
          \quad
          \big\|\partial_z\htkrnl^t\big\|_{L^2(\crc)} \le C'_\hbar t^{-3/4}
          \qquad \text{for all $t>0$ and $z\in\crc$}.
	\end{align}
\end{lemma}
\begin{proof}
  The expression on the right-hand side of \eqref{eq:dfnhtkrnl} is the Fourier series representation of the expression given in \eqref{eq:eqhtkrnl}. Indeed, for the $k$th Fourier coefficient, we obtain
  \begin{align*}
    \int_0^1 e^{-2\pi ikz}\sum_{m\in\setZ} \heat^{\hbar^2 t}(z-m)\dd z
    &= \int_\setR e^{-2\pi iky} \heat^{\hbar^2 t}(y)\dd y \\
    &= \frac1{\sqrt{4\pi\hbar^2 t}}\int_\setR  \exp\left(-\frac{y^2}{4\hbar^2 t}-2\pi iky\right) \dd y \\
    &= \exp\left(-\frac{(4\pi\hbar^2 t k)^2}{4\hbar^2 t} \right) \int_\setR \heat^{\hbar^2 t}\left(y+4\pi\hbar^2 t k i\right)\dd y
      = \exp(-(2\pi k)^2\hbar^2 t).
  \end{align*}
  The other claims follow from classical properties of the heat kernel:
 \begin{align*}
    \big\|\htkrnl^t\big\|_{L^\infty(\crc)}
    &\le \sum_{m\in\setZ} \big\|\heat^{\hbar^2 t}(\cdot-m)\big\|_{L^\infty([0,1[)}
      \le (4\pi\hbar^2t)^{-1/2}\left[1+2\sum_{m=1}^\infty \exp\left(-\frac{m^2}{4\hbar^2t}\right)\right], \\
    \big\|\partial_z\htkrnl^t\big\|_{L^2(\crc)}
    &\le \sum_{m\in\setZ} \big\|\partial_y\heat^{\hbar^2 t}(\cdot-m)\big\|_{L^2([0,1[)}
    = \big\|\partial_y\heat^{\hbar^2 t}\big\|_{L^2(\setR)},
  \end{align*}
  so in particular, the estimate on $\partial_y\heat^\sigma$ carries over to $\partial_z\htkrnl^t$.
\end{proof}
\begin{lemma}
  With $C_A$ and $C_\hbar$ from \eqref{eq:CA} and \eqref{eq:bd-htkrnl}, respectively, there is a $\gamma\ge0$ such that, for all $t\in(0,1]$,
  \begin{align}
    \label{eq:gamma}
    C_A C_\hbar \int_0^t(t-s)^{-1/2}s^{-1/4}e^{\gamma s}\dd s \le \frac{e^{\gamma t}}{2}.
  \end{align}
\end{lemma}
\begin{proof}
  To ensure the existence of $\gamma$, it suffices to observe that
  \begin{align*}
    \int_0^t(t-s)^{-1/2}s^{-1/4}e^{-\gamma(t-s)}\dd s = \int_0^1 (1-\sigma)^{-1/2}\sigma^{-1/4} t^{1/4}e^{-t(1-\sigma)}\dd\sigma
  \end{align*}
  converges to zero as $\gamma\to\infty$ because $t^{1/4}e^{-\gamma t(1-\sigma)}\le(4e\gamma(1-\sigma))^{-1/4}$ converges pointwise monotonically to zero.
\end{proof}
For the following, choose $\gamma>0$ in accordance with \eqref{eq:gamma}. Let $X$ be the space of bounded continuous functions on $(0,1]\times\crc\times\crc$; we write elements $G\in X$ in the form $G^t(x,y)$. Note that $G\in X$ need not possess a limit for $t\searrow0$. Introduce the norm
\begin{align*}
  \tnrm{G} = \sup_{t\in(0,1]} e^{-\gamma t}\max_{x,y\in\crc}|G^t(x,y)|,
\end{align*}
which makes $X$ a Banach space. Further, for a given $A\in\admA$, define the fixed point operator $\Phi_A:X\to X$ with
\begin{equation}
  \label{eq:dfn-fixedpoint}
  \begin{split}
    \big(\Phi_A[G]\big)^t(x,y) 
    &= \int_0^t \int_\crc \htkrnl^{t-s}(x-z)A(z)G^s(z,y)\dd z\dd s \\ 
    &\qquad + \int_0^t\int_\crc \htkrnl^{t-s}(x-z)A(z)\htkrnl^s(z-y) \dd z\dd s.
  \end{split}
\end{equation}
\begin{lemma}
  $\Phi_A$ is a well-defined uniform contraction on $X$, for any $A\in\admA$.
\end{lemma}
\begin{proof}
  We rewrite the integrals in the definition of $\Phi_A$ with fixed integral limits,
  \begin{align*}
    \big(\Phi_A[G]\big)^t(x,y) 
    = \int_0^1 \int_\crc t\htkrnl^{t(1-\sigma)}(x-z)A(z)\big[G^{t\sigma}(z,y)+\htkrnl^{t\sigma}(z-y)\big] \dd z\dd \sigma.
  \end{align*}
  For well-definedness, first note that for each $\sigma\in(0,1)$ and $z\in\crc$, the integrand is continuous with respect to $t\in(0,1]$ and $x,y\in\crc$. Moreover, recalling the bound \eqref{eq:bd-htkrnl}, one has
  \begin{align*}
    \big|t\htkrnl^{t(1-\sigma)}(x-z)A(z) G^{t\sigma}(z,y)\big|&\le t^{\frac12} C_\hbar \tnrm{G} e^{\gamma t \sigma} (1-\sigma)^{-\frac12}|A(z)|, \\
    \big|t\htkrnl^{t(1-\sigma)}(x-z)A(z)\htkrnl^{t\sigma}(z,y)\big| &\le C_\hbar^2 (1-\sigma)^{-\frac12}\sigma^{-\frac12}|A(z)|.
  \end{align*}
  Since $A\in L^1(\crc)$, and since $(1-\sigma)^{-\frac12}\sigma^{-\frac12}$ is integrable on $(0,1)$, dominated convergence provides continuity of the integral with respect to $t$ and $(x,y)$. Also, the integral bounds are uniform with respect to $t\in(0,1]$. This proves well-definedness of $\Phi_A$ as a map from $X$ to $X$.

  For proving contractivity, let $G_1,G_2\in X$. Then, for any $t\in(0,1)$ and $x,y\in\crc$,
  \begin{align*}
    \big|\Phi_A[G_1]^t(x,y)-\Phi_A[G_2]^t(x,y)\big|
    &\le \int_0^t \int_\crc \htkrnl^{t-s}(x-z)|A(z)||G_1^s(z,y)-G_2^s(z,y)|\dd z \dd s \\
    &\le \int_0^t C_\hbar(t-s)^{-1/2}\|A\|_{L^1}e^{\gamma s}\tnrm{G_1-G_2} \dd s \\
    &\le C_A C_\hbar \int_0^t (t-s)^{-1/2}e^{\gamma s}\dd s\, \tnrm{G_1-G_2} \\
    &\le \frac12e^{\gamma t}\tnrm{G_1-G_2},
  \end{align*}
  where we have used that $\gamma$ satisfies \eqref{eq:gamma}. Multiply by $e^{-\gamma t}$ and take the supremum to conclude that $\tnrm{\Phi_A[G_1]-\Phi_A[G_2]}\le\frac12\tnrm{G_1-G_2}$.
\end{proof}
\begin{lemma}
  \label{lem:fixed}
  For every $A\in\adm$, there is precisely one fixed point $G_A$ of $\Phi_A$, i.e., precisley one $G_A\in X$ with the property
  \begin{align}
    \label{eq:fixed}
    G_A^t(x,y) = \int_0^t \int_\crc \htkrnl^{t-s}(x-z)A(z)\big[G_A^s(z,y)+\htkrnl^s(z-y)\big]\dd z\dd s 
    \quad\text{for all $x,y\in\crc$, $t\in(0,1]$}.
  \end{align}
  We call $G_A$ the \emph{auxiliary kernel} for $A$.
\end{lemma}
\begin{proof}
  This is a consequence of the contraction mapping principle.
\end{proof}
\begin{lemma}
  \label{lem:KG}
  There some constant $K_G$ ---  uniform with respect to $A\in\adm$ --- such that the auxiliary kernel $G_A$ satisfies
  \begin{align}
    \label{eq:KG}
    \sup_{t\in(0,1]}\|G_A^t\|_{L^\infty(\crc\times\crc)} \le K_G.
  \end{align}
\end{lemma}
\begin{proof}
  We estimate inside the fixed point equation \eqref{eq:fixed}, using the property \eqref{eq:gamma} of $\gamma$, 
  \begin{align*}
    |G_A^t(x,y)|
    &\le \int_0^tC_\hbar(t-s)^{-1/2}\|A\|_{L^1(\crc)}\big\|G_A^s\big\|_{L^\infty(\crc\times\crc)} \dd s
      + \int_0^tC_\hbar(t-s)^{-1/2}\|A\|_{L^1(\crc)}C_\hbar s^{-1/2}\dd s \\
    &\le C_AC_\hbar\int_0^t(t-s)^{-1/2}e^{\gamma s}\dd s\ \tnrm{G_A} + C_AC_\hbar^2\int_0^1(1-\sigma)^{-1/2}\sigma^{-1/2}\dd\sigma \\
    &\le \frac{e^{\gamma t}}2\tnrm{G_A} + \pi C_A C_\hbar^2.
  \end{align*}
  It thus follows that $\tnrm{G_A} \le \frac12\tnrm{G_A} + \pi C_AC_\hbar^2$, which implies \eqref{eq:KG} with $K_G=2e^\gamma\pi C_AC_\hbar^2$.
\end{proof}
\begin{definition}
  We call the unique fixed point $G_A\in X$ \emph{auxiliary kernel} for $A$.
\end{definition}
\begin{lemma}
  \label{lem:max-hoelder}
  The auxiliary kernel $G_A\in X$ is H\"older continuous in the first spatial variable $x$. More precisely, there is a constant $K_H$ --- uniform with respect to $A\in\adm$ --- such that
  \begin{align}
    \label{eq:max-hoelder}
    \big|G_A^t(x',y) - G_A^t(x,y)\big| \le K_Ht^{-1/4}|x'-x|^{1/2} \quad \text{for all $x,x',y\in\crc$ and $0<t\le1$}.
  \end{align}
\end{lemma}
\begin{remark}
  The approximation result below will show that $G_A^t$ is symmetric, i.e., $G_A^t(x,y)=G_A^t(x,y)$. Thus, a posteriori, the bound \eqref{eq:max-hoelder} is even true in both spatial variables.
\end{remark}
\begin{proof}[Proof of Lemma \ref{lem:max-hoelder}]
  Observe that, by H\"older's inequality and the bound\eqref{eq:bd-htkrnl}, we have for any $\sigma>0$ and $r,r'\in\crc$:
  \begin{align*}
    \big|\htkrnl^\sigma(r')-\htkrnl^\sigma(r)\big|
    \le \left|\int_r^{r'}\partial_z\htkrnl^\sigma(r'')\dd r''\right|
    \le |r'-r|^{1/2}\left(\int_\crc|\partial_z\htkrnl^\sigma(r'')|^2\dd r''\right)^{1/2}
    \le C_\hbar'\sigma^{-3/4}|r'-r|^{1/2}.
  \end{align*}
  Substitute this with $\sigma=t-s$ and $r=x-z$, $r'=x'-z$ into the fixed point relation \eqref{eq:fixed}:
  \begin{align*}
    \big|G_A^t(x',y) - G_A^t(x,y)\big|
    &\le \int_0^t\int_\crc \big|\htkrnl^{t-s}(x'-z)-\htkrnl^{t-s}(x-z)\big|\,|A(z)|\,\big[|G_A^s(z,y)|+|\htkrnl^s(z-y)|\big]\dd z\dd s \\
    &\le \int_0^t C_\hbar'C_A|(x'-z)-(x-z)|^{1/2}(t-s)^{-3/4}\big[K_G+C_\hbar s^{-1/2}\big]\dd s \\
    &\le C_\hbar'C_A\left[4K_Gt^{1/4}+C_\hbar t^{-1/4}\int_0^1(1-\sigma)^{-3/4}\sigma^{-1/2}\dd\sigma\right]\,|x'-x|^{1/2},
  \end{align*}
  where we have used the heat kernel bound \eqref{eq:bd-htkrnl} again, and the bound \eqref{eq:KG} on $G_A$.
  Since $t\in(0,1]$, one can choose $K_H$ such that \eqref{eq:max-hoelder} holds.
\end{proof}
\begin{lemma}
  \label{lem:exp-regular}
  The map $A\mapsto G_A^1$ is Lipschitz continuous from $\adm$, with the $L^1(\crc)$-norm, into $C(\crc\times\crc)$, with the uniform norm. 
\end{lemma}
\begin{proof}
  Let $G_A$ and $G_B$ be fixed points of $\Phi_A$ and $\Phi_B$, respectively, for $A,B\in\adm$.
  Subtracting the respective fixed points relations \eqref{eq:fixed}, we obtain
  \begin{align*}
    \big[G_A^t-G_B^t\big](x,y)
    &= \int_0^t\int_\crc \htkrnl^{t-s}(x-z) \big[A(z)G_A^s(z,y)-B(z)G_B^s(z,y)\big]\dd z\dd s \\
    &=  \int_0^t\int_\crc \htkrnl^{t-s}(x-z)A(z)\big[G_A^s-G_B^s\big](z,y)\dd z\dd s \\
    &\qquad + \int_0^t\int_\crc\htkrnl^{t-s}(x-z)\big[A-B\big](z)G_B^s(z,y)\dd z\dd s.
  \end{align*}
  Therefore, using the property \eqref{eq:gamma} of $\gamma$ and the bound \eqref{eq:KG} on $G_A$,
  \begin{align*}
    \sup_{x,y\in\crc}\big|G_A^t-G_B^t\big|(x,y)
    & \le \int_0^t C_\hbar(t-s)^{-1/2} \|A\|_{L^1(\crc)}e^{\gamma s}\tnrm{G_A-G_B}\dd s \\
    &\qquad + \int_0^t C_\hbar(t-s)^{-1/2}\|A-B\|_{L^1(\crc)}\|G_B\|_{L^\infty(\crc\times\crc)}\dd s \\
    &\le \frac{e^{\gamma t}}2\tnrm{G_A-G_B} + 2C_\hbar K_Gt^{1/2}\,\|A-B\|_{L^1(\crc)}.
  \end{align*}
  Taking the supremum in $t\in(0,1]$, we obtain that
  \begin{align*}
    \frac12\tnrm{G_A-G_B} \le 2C_\hbar K_G\,\|A-B\|_{L^1(\crc)},
  \end{align*}
  and in particular, we have that
  \begin{align*}
    \sup_{x,y\in\crc}\big|G_A^1-G_B^1\big|(x,y) \le 4e^{\gamma}C_\hbar K_G\,\|A-B\|_{L^1(\crc)},
  \end{align*}
  which is the claimed Lipschitz dependence.
\end{proof}
\begin{definition}
  The \emph{quantum exponential} map $\nop:\adm\to C(\crc)$ is given by
  \begin{align}
    \label{eq:qexp}
    \nop[A](x) = G_A^1(x,x) + \htkrnl^1(0),
    \quad \text{where} \quad
    \htkrnl^1(0) = \frac1{\sqrt{4\pi\hbar^2}}\sum_{m\in\setZ} \exp\left(-\left|\frac{m}{2\hbar}\right|^2\right).
  \end{align}
\end{definition}
\begin{lemma}
  \label{lem:nopLip}
  $\nop$ is Lipschitz continuous from $\adm$, with the $L^1(\crc)$-norm, to $C(\crc)$, with the uniform norm.
\end{lemma}
\begin{proof}
  This is a direct consequence of Lemma \ref{lem:exp-regular} above.  
\end{proof}

\subsection{Relation to the operator exponential}
In this subsection, we establish the relation between the auxiliary kernel $G_A$ and the operator exponential of $\hbar^2\Delta+A$. This is to bridge the gap between the definitions of $\nop$ in \eqref{eq:qexp} above, and in \cite{Pinaud19}, respectively. We emphasize that these considerations are not relevant for the further analysis, but for the physical interpretation of our results.

\begin{definition}
  We define the \emph{solution kernel} $H_A\in C((0,1]\times\crc\times\crc)$ for $A\in\adm$ by
  \begin{align}
    \label{eq:HfromG}
    H_A^t(x,y):=G_A^t(x,y)+\htkrnl^t(x-y).
  \end{align}
\end{definition}
Formally, we have that $H^t_A$ is integral kernel for the Hilbert Schmidt operator $\exp(t[\hbar^2\Delta+A])$ on $L^2(\crc)$. While $\htkrnl\in C^\infty(\crc)$ for each $t>0$, the solution kernel $H_A^t$ is typically continuous, but not differentiable. Nonetheless, it satisfies the initial value problem
\begin{align*}
  \partial_tH^t(x,y) = \big(\hbar^2\Delta_x+A(x)\big)H^t(x,y),\quad H^0(x,y)=\delta(x-y)
\end{align*}
in distributional sense. More precisely:
\begin{corollary}
  \label{cor:pinaud}
  For each $\varphi\in C^2(\crc)$, the function $h\in C((0,1]\times\crc)$ with
  \begin{align*}
    h^t(y) = \int_\crc \varphi(x)H_A^t(x,y)\dd x
  \end{align*}
  is continuously differentiable in $t\in(0,1]$, for each $y\in\crc$, and satisfies
  \begin{align}
    \label{eq:weakpde}
    \partial_th^t(y) = \int_\crc \big(\partial_{xx}\varphi(x)+A(x)\varphi(x)\big)H_A^t(x,y)\dd x,
    \quad
    \lim_{t\searrow0} h^t(y) = \varphi(y).
  \end{align}
\end{corollary}
\begin{proof}
  Consider the auxiliary function $g\in C((0,1]\times\crc)$ given by
  \begin{align*}
    g^t(y) = \int_\crc \varphi(x)G_A^t(x,y)\dd x,
  \end{align*}
  which, by the fixed point relation \eqref{eq:fixed}, can be written as
  \begin{align*}
    g^t(y) = \int_0^t\int_\crc \psi^{t-s}(z)A(z)\big[G_A^s(z,y)+\htkrnl^s(z-y)\big]\dd z\dd s,
    \quad\text{with}\quad
    \psi^{t-s}(z) = \int_\crc \varphi(x)\htkrnl^{t-s}(z-x)\dd x.
  \end{align*}
  SInce $\htkrnl^{t-s}(\cdot)$ is of unit mass, $|\psi^{t-s}(z)|\le\max|\varphi|$ for all $z\in\crc$, $s\in[0,t)$ and $t\in(0,1]$. From property \eqref{eq:ic-htkrnl}, we further know that $\psi^{t-s}(z)\to\varphi(z)$ as $s\nearrow t$, uniformly in $z\in\crc$. Moreover, by \eqref{eq:pde-htkrnl},
  \begin{align*}
    \partial_t\psi^{t-s}(z) = \int_\crc \varphi(x)\,\partial_{xx}\htkrnl^{t-s}(z-x)\dd x = \int_\crc \partial_{xx}\varphi(x)\,\htkrnl^{t-s}(z-x)\dd x.
  \end{align*}
  Now by boundedness \eqref{eq:KG} of $G_A$ and by the estimates \eqref{eq:bd-htkrnl} on $\htkrnl^s$, it follows that $t\mapsto g^t(y)$ is differentiable with respect to $t\in(0,1]$ at any $y\in\crc$, with
  \begin{align*}
    \partial_tg^t(y)
    &= \int_\crc\varphi(z)A(z)\big[G_A^t(z,y)+\htkrnl^t(z-y)\big]\dd z \\
    &\qquad + \int_0^t\int_\crc\left(\int_\crc \partial_{xx}\varphi(x)\,\htkrnl^{t-s}(z-x)\dd x\right)\,\big[G_A^s(z,y) + \htkrnl^s(z-y)\big]\dd z\dd s \\
    &= \int_\crc \varphi(z)A(z)H_A^t(z,y)\dd z
      + \int_\crc \partial_{xx}\varphi(x)\,\Phi_A^t\big[G_A\big](x,y)\dd x \\
    &= \int_\crc \big(\partial_{xx}\varphi(x)+A(x)\varphi(x)\big)H_A^t(x,y)\dd x - \int_\crc\partial_{xx}\varphi(x)\,\htkrnl^t(y-x)\dd x,
  \end{align*}
  where we have used the fixed point property \eqref{eq:fixed} of $G_A$.
  Addition of
  \begin{align*}
    \partial_t\int_\crc\varphi(x) \htkrnl^t(x-y)\dd x = \int_\crc\partial_{xx}\varphi(x)\,\htkrnl^t(x-y)\dd y,
  \end{align*}
  which follows from \eqref{eq:pde-htkrnl}, we finally obtain the differential equation in \eqref{eq:weakpde}. To prove attainment of the initial condition as well, it suffices to show that $g^t(y)\to 0$ as $t\searrow0$, for every $y\in\crc$. By uniform boundedness of $\psi$ and $G_A$, and the estimates \eqref{eq:bd-htkrnl} on $\htkrnl$, we obtain
  \begin{align*}
    |g^t(y)| \le \int_0^t \max|\varphi|\,\|A\|_{L^1(\crc)}\big(K_G+C_\hbar(t-s)^{-1/2}\big)\dd s \le Ct^{1/2},
  \end{align*}
  proving convergence of $g^t(y)$ to zero. Combine this with \eqref{eq:ic-htkrnl} to finish the proof.
\end{proof}

\subsection{Approximation by the discrete quantum exponential}
Recall the definition of the discrete quantum exponential map $\nop_\delta$ in Proposition \ref{th:quantmax}: up to scaling, $\nop_\delta[A]$ is the diagonal of the discrete quantum Maxwellian matrix $\exp(\hbar^2\Delta_\delta+A)$, which parallels the original definition of $\nop$ in \cite{Pinaud19}. For a quantitative comparision of $\nop_\delta$ to $\nop$, we first define a discrete analogue $G^t_{A,\delta}$ of $G_A^t$, show that it satisfies a fixed point relation similar to \eqref{eq:fixed}, and then estimate the difference between $G^t_{A,\delta}$ and $G_A^t$.
\smallskip

Below, we consider potentials $A$ in the intersection of $\adm$ with $\Lp 1$, e.g., in
\begin{align*}
  \adm_\delta := \left\{ A\in \Lp 1\,\middle|\, \delta\sum_{\xi\in\crcN}|A(\xi)| \le C_A \right\},
\end{align*}
with the same constant $C_A$ as above.
Further, we shall use the following sloppy notation for sums of precisely $N$ terms:
\begin{align}
  \label{eq:sloppy}
  \sum_{|k|\lesssim N/2} a_k =
  \begin{cases}
    \sum_{k=-(N/2-1)}^{N/2} a_k & \text{if $N$ is even}, \\
    \sum_{k=-(N-1)/2}^{(N-1)/2} a_k & \text{if $N$ is odd}.
  \end{cases}
\end{align}
In analogy to \eqref{eq:dfnhtkrnl}, we define the discrete heat kernel as follows.
\begin{definition}
  At each $t>0$, the \emph{discrete heat kernel} $\htkrnl_\delta^t$ is given by
  \begin{align}
    \label{eq:dfndhtkrnl}
    \htkrnl_\delta^t(\zeta) = \sum_{|k|\lesssim N/2} e^{-\omega_k\hbar^2 t}e^{2\pi i k\zeta}
    \quad \text{for all $\zeta\in\crcN$},
  \end{align}
  where the coefficients $\omega_k$ are the eigenvalues of $-\Delta_\delta$, see \eqref{eq:DeltaEV}.
\end{definition}
\begin{lemma}
  $\htkrnl_\delta$ is the solution kernel for the one-periodic heat equation, i.e., 
  \begin{align}
    \label{eq:pde-dhtkrnl}
    \partial_t\htkrnl_{\delta}^t(\zeta)
    = \hbar^2\Delta_\delta \htkrnl_\delta^t(\zeta),
    \qquad
    \htkrnl_{\delta}^0(\zeta)
    = \begin{cases} \delta^{-1} & \text{for $\zeta=0$}, \\ 0 & \text{for $\zeta\neq0$}. \end{cases}
  \end{align}
  Moreover, there is a $\delta$-uniform constant $C_\hbar$ such that, for all $t\in(0,1]$ and $\zeta\in\crcN$,
  \begin{align}
    \label{eq:bd-dhtkrnl}
    \big|\htkrnl_\delta^t(\zeta)\big| \le C_\hbar t^{-1/2}.
  \end{align}
\end{lemma}
\begin{remark}
  Additionally, one has that  $\htkrnl^t_\delta$ is non-negative and of unit mass, i.e., $\delta\sum_{\xi\in\crcN}\htkrnl^t_\delta(\xi)=1$. The former is difficult conclude from \eqref{eq:dfndhtkrnl}, the latter simply follows from $\omega_0=0$. These properties are not needed for our analysis.
\end{remark}
\begin{proof}
  Recall that the $\omega_k$ are eigenvalues of $-\Delta_\delta$, with associated eigenvectors $w_k$ such that $\cmp{w_k}j=e^{2\pi i k j\delta}$, see Lemma \ref{lem:DeltaEV} in Appendix \ref{app:Laplace}. It then follows that $\htkrnl_\delta^{(\cdot)}$ is the kernel for the operator exponential of $\hbar^2\Delta_\delta$. This implies \eqref{eq:pde-dhtkrnl}.
  Concerning the estimate \eqref{eq:bd-dhtkrnl}, simply observe that by means of estimate \eqref{eq:LT}, we have that
  \begin{align*}
    \big|\htkrnl_\delta^t(\zeta)\big|
    \le \sum_{|k|\lesssim N/2} e^{-\omega_k\hbar^2 t}\big|e^{2\pi i k\zeta}\big|
    = \sum_{|k|\lesssim N/2} e^{-\omega_k\hbar^2 t}
    \le 1+ \frac{\sqrt\pi}4(\hbar^2 t)^{-1/2}.
  \end{align*}
  Choosing $C_\hbar$ appropriately, \eqref{eq:bd-dhtkrnl} follows for all $t\in(0,1]$.
\end{proof}
\begin{definition}
  For each $t\in[0,1]$, the \emph{discrete solution kernel} $H_{A,\delta}$ for $A\in\adm_\delta$ is given by $\delta H_{A,\delta}^t(k\delta,\ell\delta)=\exp(t[\hbar^2\Delta_\delta+A])_{k,\ell}$.
\end{definition}
\begin{lemma}
  The discrete solution kernel satisfies 
  \begin{align}
    \label{eq:ode1-sol}
    \partial_tH_{A,\delta}^t(\xi,\eta)
    &= \hbar^2\Delta_\delta H_{A,\delta}^t(\xi,\eta) + A(\xi)H_{A,\delta}^t(\xi,\eta), \\
    \label{eq:ode2-sol}
    H_{A,\delta}^0(\xi,\eta)
    &= \begin{cases} \delta^{-1} & \text{for $\xi=\eta$}, \\ 0 & \text{for $\xi\neq\eta$}. \end{cases}
  \end{align}
  Consequently, the discrete quantum exponential $\nop_\delta[A]\in \Lp 2$ of $A\in\adm_\delta$ is given by
  \begin{align}
    \label{eq:dqexp}
    \nop_\delta[A](\xi) = H_{A,\delta}^1(\xi,\xi) .
  \end{align}
\end{lemma}
\begin{proof}
  This is a direct consequence from the definition of the matrix exponential, and the definition on $\nop_\delta$ in Proposition \ref{th:quantmax}.
\end{proof}
\begin{lemma}
  \label{lem:nopmonotone}
  $\nop_\delta$ is a monotone map from $\Lp 2$ to itself.
\end{lemma}
\begin{proof}
  We need to show that, for $A,A'\in \Lp 2$,
  \begin{align*}
    J(A,A'):=\int_\crc \big(\nop_\delta[A]-\nop_\delta[A']\big)(A-A')\dd x \ge 0.
  \end{align*}
  By definition of $\nop_\delta$ as the diagonal elements of the matrix exponential of $\hbar^2\Delta_\delta+A$, this expression can be rewritten as the trace of a matrix,
  \begin{align*}
    J(A,A') = \tr{\big(\exp(\hbar^2\Delta_\delta+A)-\exp(\hbar^2\Delta_\delta+{A'})\big)\big((\hbar^2\Delta_\delta+A)-(\hbar^2\Delta_\delta+{A'})\big)}.
  \end{align*}
  This expression is non-negative because of Corollary \ref{cor:exppos}.
\end{proof}
\begin{definition}
  We define the \emph{discrete auxiliary kernel} $G_{A,\delta}$ for $A\in\adm_\delta$ by
  \begin{align*}
    G_{A,\delta}^t(\xi,\eta):=H_{A,\delta}^t(\xi,\eta)-\htkrnl_\delta^t(\xi-\eta).
  \end{align*}
\end{definition}
\begin{lemma}
  The auxiliary kernel  $G_{A,\delta}$ satisfies, at each $t>0$,
  \begin{align}
    \label{eq:dfixed}
    G_{A,\delta}^t(\xi,\eta)
    = \int_0^t \delta\sum_{\zeta\in\crcN}\htkrnl_\delta^{t-s}(\xi-\zeta)A(\zeta)\big[G_{A,\delta}^s(\zeta,\eta)+\htkrnl_\delta^s(\zeta-\eta)\big]\dd s.
  \end{align}
\end{lemma}
\begin{proof}
  Comparing \eqref{eq:ode1-sol}\&\eqref{eq:ode2-sol} to \eqref{eq:pde-dhtkrnl}, one easily verifies that \eqref{eq:dfixed} is a consequence of the variation of constant formula for ODEs.
\end{proof}
\begin{lemma}
  \label{lem:auxkernelbound}
  There a constant $K_G$, which is uniform in $\delta$ and $A\in\adm_\delta$, such that
  \begin{align*}
    \sup_{t\in(0,1]}\max_{\xi,\eta\in\crcN}|G_{A,\delta}^t(\xi,\eta)|\le K_G.
  \end{align*}
\end{lemma}
\begin{proof}
  We omit the proof here, which can be peformed in complete analogy to the one for Lemma \ref{lem:KG} in the non-discretized setting.
  One uses \eqref{eq:dfixed} above in place of the fixed point relation \eqref{eq:fixed}.
  Note that the auxiliary quantity $\gamma$ can be introduced in the same way as in \eqref{eq:gamma}, using the discrete heat kernel estimate in \eqref{eq:bd-dhtkrnl} instead of the continuous one in \eqref{eq:bd-htkrnl}.
\end{proof}
Our key estimate for the subsequent approximation is:
\begin{proposition}
  \label{prp:prekernel}
  There is a constant $C_K>0$ such that
  \begin{align}
    \label{eq:prekernel2}
    \max_{\zeta\in\crcN}\left|\fint_{I_\zeta}\htkrnl^t(z)\dd z - \htkrnl^t(\zeta)\right|\le C_K\delta^{1/4}t^{-3/4},
  \end{align}
  uniformly in $\delta$ and $t\in(0,1]$.  
\end{proposition}
The proof is given in Appendix \ref{app:A}. Our main result in this section is the following.
\begin{proposition}
  \label{prp:holy}
  There is a constant $K_*$, such that
  \begin{align}
    \label{eq:holyest}
    \big|G_A^t(\xi,\eta)-G_{A,\delta}^t(\xi,\eta)\big| \le K_*\delta^{1/4}t^{-1/4}
    \quad \text{for all $\xi,\eta\in\crcN$ and $t\in(0,1]$}.
  \end{align}
  The constant $K_*$ is uniform with respect to $A\in\adm$ and the discretization $\delta$.
\end{proposition}
\begin{proof}
  Throughout this proof, the discretization parameter $\delta$ and the potential $A\in\adm_\delta$ are fixed; we shall omit the sub-index $A$ on $G_A$ and $G_{\delta,A}$ from now on. By the properties of $G$ and $G_\delta$, the quantity
  \begin{align*}
    \bnrm{G-G_\delta} = \sup_{t\in(0,1]} t^{1/4}e^{-\gamma t}\max_{\xi,\eta\in\crcN}\big|G^t(\xi,\eta)-G_\delta^t(\xi,\eta)\big|
  \end{align*}
  is clearly finite. The goal of the following is to derive a bound on $\bnrm{G-G_\delta}$ that is independent of $A\in\adm_\delta$ and $\delta$.
  
  To begin with, subtract the variation of constant formula \eqref{eq:dfixed} for $G_\delta$ from the fixed point relation \eqref{eq:fixed} for $G$, both evaluated at the same $t\in(0,1]$ and $\xi,\eta\in\crcN$. Recalling that $A$ is constant on the intervals $I_\zeta$, one obtains:
  \begin{align}
    \label{eq:holydiff}
    G^t(\xi,\eta) - G_\delta^t(\xi,\eta) = \int_0^t\delta\sum_\zeta A(\zeta) \big[Q_1^{s,t}(\xi,\eta;\zeta)+Q_2^{s,t}(\xi,\eta;\zeta)\big]\dd s,
  \end{align}
  where
  \begin{align*}
    Q_1^{s,t}(\xi,\eta;\zeta) &= \fint_{I_\zeta}\htkrnl^{t-s}(\xi-z)G^s(z,\eta)\dd z - \htkrnl_\delta^{t-s}(\xi-\zeta)G_\delta^s(\zeta,\eta), \\
    Q_2^{s,t}(\xi,\eta;\zeta) &= \fint_{I_\zeta}\htkrnl^{t-s}(\xi-z)\htkrnl^s(z-\eta)\dd z - \htkrnl_\delta^{t-s}(\xi-\zeta)\htkrnl_\delta^s(\zeta,\eta).
  \end{align*}
  For further estimation of $Q_1$ and $Q_2$, rewrite these as follows:
  \begin{align*}
    Q_1^{s,t}(\xi,\eta;\zeta)
    &= \fint_{I_\zeta}\big[\htkrnl^{t-s}(\xi-z)-\htkrnl_\delta^{t-s}(\xi-\zeta)\big]G^s(z,\eta)\dd z \\
    &\qquad + \htkrnl_\delta^{t-s}(\xi-\zeta) \left(\fint_{I_\zeta}G^s(z,\eta)\dd z-G^s(\zeta,\eta) \right) \\
    &\qquad + \htkrnl_\delta^{t-s}(\xi-\zeta)\big[G^s(\zeta,\eta)-G^s_\delta(\zeta,\eta)\big]
    \\
    Q_2^{s,t}(\xi,\eta;\zeta)
    &= \fint_{I_\zeta}\big[\htkrnl^{t-s}(\xi-z)-\htkrnl_\delta^{t-s}(\xi-\zeta)\big]\htkrnl^s(z-\eta)\dd z \\
    &\qquad + \htkrnl_\delta^{t-s}(\xi-\zeta)\left(\fint_{I_\zeta}\htkrnl^s(z-\eta)\dd z-\htkrnl_\delta^s(\zeta-\eta)\right) .
  \end{align*}
  With the aid of \eqref{eq:bd-htkrnl}, \eqref{eq:KG}, \eqref{eq:max-hoelder}, and \eqref{eq:prekernel2}, any by definition of $\bnrm{G-G_\delta}$, it follows that
  \begin{align*}
    |Q_1^{s,t}(\xi,\eta;\zeta)|
    &\le C_\hbar(t-s)^{-1/2}\sup_{z\in I_\zeta}\left|G^s(z,\eta)\dd z-G^s(\zeta,\eta)\right| \\
    &\qquad + \left|\fint_{I_\zeta}\htkrnl^{t-s}(\xi-z)\dd z - \htkrnl_\delta^{t-s}(\xi-\zeta)\right|\sup_{z\in I_\zeta}|G^s(z,\eta)| \\
    &\qquad + C_\hbar(t-s)^{-1/2}\big|G^s(\zeta,\eta)-G^s_\delta(\zeta,\eta)\big| \\
    &\le C_\hbar(t-s)^{-1/2}\,K_H\delta^{1/2}s^{-1/4}
      + C_K\delta^{1/4}(t-s)^{-3/4}\,K_G \\
    &\qquad + C_\hbar(t-s)^{-1/2}\,\bnrm{G-G_\delta}s^{-1/4}e^{\gamma s}.
  \end{align*}
  In a similar manner, now using just \eqref{eq:bd-htkrnl} and \eqref{eq:prekernel2}, it follows that
  \begin{align*}
    |Q_2^{s,t}(\xi,\eta;\zeta)|
    &\le  \left|\fint_{I_\zeta}\htkrnl^{t-s}(\xi-z)\dd z - \htkrnl_\delta^{t-s}(\xi-\zeta)\right|\,C_\hbar s^{-1/2} \\
    &\qquad + C_\hbar(t-s)^{-1/2}\left|\fint_{I_\zeta}\htkrnl^s(\zeta-z)\dd z-\htkrnl_\delta^s(\zeta-\eta)\right| \\
    &\le C_K\delta^{1/4}(t-s)^{-3/4}\,C_\hbar s^{-1/2} + C_\hbar(t-s)^{-1/2}\,C_K\delta^{1/4}s^{-3/4}.
  \end{align*}
  Substitution into \eqref{eq:holydiff} and recalling the relation \eqref{eq:gamma} for $\gamma$ yields
  \begin{align*}
    \big|G^t(\xi,\eta) - G_\delta^t(\xi,\eta) \big|
    &\le \frac{e^{\gamma t}}{2} K_H \delta^{1/2}
      + 4C_KK_G\ \delta^{1/4}t^{1/4}
      + \frac{e^{\gamma t}}{2} \,\bnrm{G-G_\delta} \\
    & + 2C_A C_\hbar M\left(\int_0^1\sigma^{-1/2}(1-\sigma)^{-3/4}\dd\sigma\right)\ \delta^{1/4}t^{-1/4},
  \end{align*}
  and so, finally
  \begin{align*}
    \bnrm{G-G_\delta} \le K\delta^{1/4} + \frac12 \bnrm{G-G_\delta},
  \end{align*}
  with a constant $K$ that is uniform in $A\in\adm_\delta$ and $\delta\in(0,1)$. This implies \eqref{eq:holyest}.
\end{proof}

\subsection{Continuous limit of the kernels}
Below, we summarize implications of Proposition \ref{prp:holy} above. Throughout this section, we assume that a sequence of step width $\delta$ is given and study the limit $\delta\searrow0$ along that sequence. By abuse of notation, we identify functions $f:\crcN\to\setR$ with their piecewise constant reconstruction $\tilde f:\crc\to\setR$, i.e., $\tilde f(x)=f(\xi)$ for all $x\in I_\xi$.
\begin{corollary}
  \label{cor:staticconvergence}
  Let a sequence of potentials $A_\delta\in \Lp 1$ be given.
  If $A_\delta$ is bounded in $\Lp 1$, then $\nop_\delta[A_\delta]$ is bounded in $\Lp\infty$.
  And if $A_\delta\to A_*$ in $L^1(\crc)$, then
  \begin{align}
    \label{eq:nop2nop}
    \nop_{\delta}[A_\delta] \to \nop[A_*] \quad \text{uniformly on $\crc$}.
  \end{align}
\end{corollary}
\begin{proof}
  Recall the definitions of the non-discrete and discrete quantum exponentials in \eqref{eq:qexp} and \eqref{eq:dqexp}, respectively.
  The claim about boundedness is a consequence of Lemma \ref{lem:auxkernelbound}.
  For the proof of convergence \eqref{eq:nop2nop}, it suffices to verify that $\htkrnl_{\delta}^1(0)\to\htkrnl^1(0)$ --- which is a direct consequence of \eqref{eq:prekernel2} --- on the one hand, and that $G_{\delta,A_\delta}^1\to G_{A_*}^1$ uniformly on $\crc\times\crc$ on the other hand. Concerning the latter:
  \begin{align*}
    \big\|G^1_{A_\delta,\delta}-G^1_{A_*}\big\|_{L^\infty(\crc)} \le \big\|G^1_{A_\delta,\delta}-G^1_{A_\delta}\big\|_{L^\infty(\crc)} + \big\|G^1_{A_\delta}-G^1_{A_*}\big\|_{L^\infty(\crc)} .
  \end{align*}
  Both differences on the right-hand side converge to zero, the first because of \eqref{eq:holyest} for $t=1$, the second because of the Lipschitz continuity stated in Lemma \ref{lem:nopLip}.
\end{proof}
The next result is the key ingredient in the proof of Proposition \ref{prp:theendofQDD} in the next section.
\begin{corollary}
  \label{cor:goal}
  Let a sequence of time-dependent densities $n_\delta:[0,T]\to \Lp \infty$ with associated time-dependent potentials $A_\delta:[0,T]\to \Lp 1$ be given, i.e., $n_\delta(t;\cdot)=\nop_\delta[A_\delta(t;\cdot)]$ for almost every $t\in[0,T]$. Assume that, as $\delta\searrow0$,
  \begin{itemize}
  \item $n_\delta$ converges to a limit $n_*$ uniformly on $[0,T]\times\crc$,
  \item $A_\delta$ converges to a limit $A_*$ weakly in $L^2([0,T]\times\crc)$.
  \end{itemize}
  Then $n_*=\nop[A_*]$.
\end{corollary}
\begin{proof}
  Fix some $C_A>0$. Let
  \begin{align*}
    J:= \big\{t\in[0,T]\,\big|\,\|A_*(t;\cdot)\|_{L^1(\crc)}\le C_A\big\}.
  \end{align*}
  Below, we are going to show that
  \begin{align}
    \label{eq:preidentify}
    n_*(t;\cdot)=\nop[A_*(t;\cdot)] \quad \text{for a.e. $t\in J$}.
  \end{align} 
  Since the sets $J$ are increasing as $C_A$ grows and exhaust almost all of $[0,T]$ as $C_A\nearrow\infty$, the claim of the corollary follows.

  First observe that $\nop[A_*]\in L^\infty(J\times\crc)$, see Lemma \ref{lem:KG}.
  For each $h\in(0,1)$, define $A^h\in L^2(J\times\crc)$ by
  \begin{align}
    \label{eq:Ah}
    A^h(t;\cdot) := A_*(t;\cdot) + h\big(n_*(t;\cdot)-\nop[A_*(t;\cdot)]\big)
    \quad \text{for a.e. $t\in J$}.
  \end{align}
  Next, for each $\delta$, define a (piecewise constant in space) potential $A^h_\delta\in L^2(J\times\crcN)$ such that $A^h_\delta(t;\cdot)\in \Lp 2$ is obtained from $A^h(t;\cdot)$ by  averaging over the respective intervals in $\crcN$. Observe that $\|A^h_\delta(t;\cdot)\|_{\Lp 2}\le\|A^h(t;\cdot)\|_{L^2(\crc)}$ by Jensen's inequality, and that $A^h_\delta(t;\cdot)\to A^h(t;\cdot)$ in $L^2(\crc)$ as $\delta\searrow0$, for a.e. $t\in J$. On the one hand, by dominated converges, it follows that $A^h_\delta\to A^h$ in $L^2(J\times\crc)$. On the other hand, it follows by means of \eqref{eq:nop2nop} that also $\nop_\delta[A^h_\delta(t;\cdot)]\to\nop[A^h(t;\cdot)]$ uniformly on $\crc$, for almost every $t\in J$. In particular, $\nop_\delta[A^h_\delta]$ converges to $\nop[A^h]$ pointwise almost everywhere on $J\times\crc$. Since $\|\nop_\delta[A^h_\delta]\|_{L^\infty(J\times\crcN)}$ is $\delta$-uniformly bounded, see Corollary \ref{cor:staticconvergence} above, dominated convergence implies  $\nop_\delta[A^h_\delta]\to\nop[A^h]$ in $L^2(J\times\crc)$.

  Recall that the operators $\nop_\delta$ are monotone, see Lemma \ref{lem:nopmonotone}, and therefore
  \begin{align*}
    \int_J\int_{\crc} \big(n_\delta-\nop_\delta[A^h_\delta]\big)\big(A_\delta-A^h_\delta)\dd x\dd t \ge 0,
  \end{align*}
  for each $\delta$. By the hypotheses and the considerations above, $n_\delta-\nop_\delta[A^h_\delta]$ converges to $n_*-\nop[A^h]$ strongly in $L^2(J\times\crc)$, and $A_\delta-A^h_\delta$ converges to $A_*-A^h=-h(n_*-\nop[A_*])$ weakly in $L^2(J\times\crc)$, repectively. Therefore, after division by $-h<0$,
  \begin{align*}
    \int_J\int_\crc \big(n_*-\nop[A^h]\big)\big(n_*-\nop[A_*]\big)\dd x\dd t \le 0.
  \end{align*}
  It remains to perform the limit $h\searrow0$. As $A^h(t;\cdot)\to A_*(t;\cdot)$ in $L^2(\crc)$ for almost every $t\in J$ by definition in \eqref{eq:Ah}, it follows by Lipschitz continuity of $\nop$, see Lemma \ref{lem:nopLip}, that also $\nop[A^h(t;\cdot)]\to\nop[A_*(t;\cdot)]$ in $L^\infty(\crc)$, and in particular that $\nop[A^h]$ converges to $\nop[A_*]$ pointwise almost everywhere on $J\times\crc$. And since $\|\nop[A^h]\|_{L^\infty(J\times\crc)}$ is $h$-uniformly bounded, one concludes by dominated convergence that $\nop[A^h]\to\nop[A_*]$ in $L^2(J\times\crc)$. Thus
  \begin{align*}
    \int_J\int_\crc \big(n_*(t;x)-\nop\big[A_*(t;x)\big]\big)^2\dd x\dd t \le 0.    
  \end{align*}
  This implies \eqref{eq:preidentify}.
\end{proof}

\section{The continuous limit}
\label{sct:convergence}

In this section we prove Theorem \ref{thm:3} about the continuous limit $\delta\to0$ of the spatially discrete nlQDD \eqref{eq:DQDD}.
Below, we shall again abuse notation and identify functions $f:\crcN\to\setR$ with their piecewise constant reconstruction $\tilde f:\crc\to\setR$, i.e., $\tilde f(x)=f(\xi)$ for all $x\in I_\xi$.

\subsection{Setup}

Throughout this section, let a sequence $\delta=1/N\searrow0$ of discretizations is fixed.
Further, fix a positive initial datum $n_0^\text{in}\in L^1(\crc)$ for \eqref{eq:QDD}, which is of unit mass and finite Fisher information $\|\partial_x\sqrt{n_0^\text{in}}\|_{L^2(\crc)}<\infty$.
For each $\delta$ under consideration, define a corresponding discretized intial datum $n_\delta^\text{in}\in\dprbp$ for \eqref{eq:DQDD0}
such that $n_\delta^\text{in}\to n_0^\text{in}$ in $L^1(\crc$), and such that the initial entropy is $\delta$-uniformly bounded,
\begin{align}
  \label{eq:overdent}
  \overline{\dent}:=  \sup_\delta \dent(n_\delta^\text{in}) - \underline{\dent}_{\hbar^2/2} < \infty.
\end{align}
Above, $\underline{\dent}_{\hbar^2/2}$ is the lower bound on $\dent$ from \eqref{eq:dentbelow}.
\begin{lemma}
  Under the above hypotheses on $n_0^\text{in}$, define the sequence $n_\delta^\text{in}\in\dprbp$ by averaging over cells,
  \begin{align*}
    n_\delta^\text{in}(\xi) = \fint_{I_\xi}n_0^\text{in}(x)\dd x.
  \end{align*}  
  Then \eqref{eq:overdent} holds.
\end{lemma}
\begin{proof}
  The representation \eqref{eq:nA} of the entropy can be rephrased as follows:
  \begin{align}
    \label{eq:dentR}
    \dent(n_\delta^\text{in}) = \min\left\{ \tr{R\log R-\hbar^2\Delta_\delta R}\,\middle|\,R\in\dens,\,R_{jj}=\delta\, n_\delta^\text{in}(j\delta)\right\}.
  \end{align}
  The idea is to choose an admissible trial matrix $\hat R\in\dens$ that produces an $\delta$-uniform bound on the right-hand side above.
  Specifically, let $w\in\setR^N$ be given by $w_j:=\sqrt{n_\delta^\text{in}(j\delta)}$, and define $\hat R:=\delta ww^T\in\dens$, i.e., $\hat R_{jk}=\delta w_jw_k$.
  Now, on the one hand, $\hat R\log\hat R=0$ because $\hat R$ has zero and one as its only eigenvalues.
  And on the other hand, using Lemma \ref{lem:fisher} from the appendix,
  \begin{align*}
    -\tr{\Delta_\delta\hat R}
    = -\delta\tr{\dq_\delta^-\dq_\delta^+ww^T} 
    &= \delta\tr{(\dq_\delta^+w)(\dq_\delta^+w)^T} \\
    &= \delta\sum_\xi\big(\dq_\delta^+\sqrt{n_\delta^\text{in}}\big)^2 
    \le 8 \big\|\partial_x\sqrt{n_0^\text{in}}\big\|_{L^2(\crc)}^2.
  \end{align*}
  Substitute $\hat R$ as trial matrix in the right-hand side of \eqref{eq:dentR} to obtain:
  \begin{align*}
    \dent(n_\delta^\text{in})
    \le \tr{\hat R\log\hat R-\hbar^2\Delta_\delta\hat R}
    \le 0 + 8\hbar^2\big\|\partial_x\sqrt{n_0^\text{in}}\big\|_{L^2(\crc)}^2,
  \end{align*}
  which is independent of $\delta$.
\end{proof}
Finally, let $n_\delta:\setRnn\to\dprbp$ be the (global and positive) solutions to \eqref{eq:DQDD} with initial datum $n_\delta(0)=n_\delta^\text{in}$.

\subsection{A priori estimates}
There are two main estimates for the passage to the continuum limit. The first is the $L^2$-bound on the product $\sqrt{\nu_\delta^+}\,\dq_\delta^+A_\delta$ in space and time from \eqref{eq:dentdiss}, which has been obtained directly from the dissipation of the entropy $\dent(n)$. The second estimate is on $n$ in $H^1$, uniformly in time, and follows directly from the monotonicity of the entropy $\dent(n)$, which controls in a subtle way non-local features of $n$. Note that this is in sharp contrast to the formal semi-classical limit $\partial_tn_\delta=\Delta_\delta n_\delta$, in which $\dent$ turns into the much simple functional $\delta\sum_\xi n\log n$, which only provides a control in $L\log L$. For later reference, we summarize the two estimates below.
\begin{proposition}\label{prop:diss}
    For any $T>0$,
    \begin{align}
      \label{est:dissip}
      \int_0^T \delta\sum_\xi \nu_\delta^+\, \big(\dq_\delta^+A_\delta\big)^2 \dd t &\le \overline{\dent}, \\
      \label{est:mono}
      \|\dq_\delta^+n_\delta\|_{L^2_\delta(\crc)}(T)&\le C(1+\overline{\dent}).
    \end{align}
\end{proposition}
The first estimate \eqref{est:dissip} has already been obtained in Proposition \ref{prp:dentdiss}.
In preparation of the proof of \eqref{est:mono}, we introduce an alternative representation of a quantum Maxwellian $\mm_\delta[n]$ for a given $n\in\dprbp$:
first, let $v_1,\ldots,v_N\in\setR^N$ be an orthonormal basis of real eigenvectors, i.e., $v_k^Tv_\ell=0$ for $k\neq\ell$ and $v_k^Tv_k=1$, with corresponding eigenvalues $\rho_k\in(0,1)$, that is
\begin{align}
  \label{eq:thisismm}
  \mm_\delta[n] = \sum_{k=1}^N\rho_kv_kv_k^T. 
\end{align}
Since $\mm_\delta[n]$ has the values of $\delta n$ on its diagonal, it follows that
\begin{align}
  \label{eq:thisisalmostn}
  \delta n(j\delta) = \mm_\delta[n]_{jj} = \sum_{k=1}^N\rho_k \cmp{v_k}j^2 .
\end{align}
Now, introduce an orthonormal basis of real eigenfunctions $\phi_1,\ldots,\phi_N\in\Lp 2$, by $\phi_k(j\delta)=\delta^{-1/2}\cmp{v_k}j$, i.e., $\delta\sum_\xi\phi_k\phi_\ell=0$ for $k\neq\ell$ and $\|\phi_k\|_{\Lp 2}=1$. This turns \eqref{eq:thisisalmostn} into
\begin{align}
  \label{eq:thisisn}
  n(\xi) = \sum_{k=1}^N\rho_k\phi_k(\xi)^2.
\end{align}
The key observation for the proof of\eqref{est:mono} is the following.
\begin{lemma}
  \label{lem:phiH1up}
  The normalized eigenfunctions $\phi_k$ in the representation \eqref{eq:thisisn} are controlled as follows:
  \begin{align}
    \label{eq:phiH1up}
    \sum_{k=1}^N\rho_k\big\|\dq^+\phi_k\big\|_{\Lp 2}^2
    \le \frac{2}{\hbar^2}\big(\dent(n)-\underline\dent_{\hbar^2/2}\big).
  \end{align}
\end{lemma}
\begin{proof}[Proof of Lemma \ref{lem:phiH1up}]
  By cyclicity of the trace, recalling that $(\dq_\delta^-)^T=-\dq_\delta^+$, one obtains that
  \begin{align*}
    -\tr{\Delta_\delta\mm_\delta[n]}
    &= -\sum_{k=1}^N\rho_k\tr{\dq_\delta^-\dq_\delta^+(\delta\phi_k\phi_k^T)} \\
    &= \sum_{k=1}^N\rho_k\delta\tr{(\dq_\delta^+\phi_k)(\dq_\delta^+\phi_k)^T}
    = \sum_{k=1}^N\rho_k\big\|\dq_\delta^+\phi_k\big\|_{\Lp 2}^2.
  \end{align*}
  Now observe that the left-hand side above is controlled by means of \eqref{eq:H1byH}, leading directly to \eqref{eq:phiH1up}.
\end{proof}
\begin{proof}[Proof of estimate \eqref{est:mono}]
  By taking difference quotients in the representation \eqref{eq:thisisn}, one obtains that
  \begin{align*}
    \dq_\delta^+n(\xi)
    = \sum_{k=1}^N\rho_k \frac{\phi_k(\xi+\delta)^2-\phi_k(\xi)^2}\delta
    = \sum_{k=1}^N\rho_k\big(\phi_k(\xi+\delta)+\phi_k(\xi)\big)\dq_\delta^+\phi_k(\xi),
  \end{align*}
  and therefore, with the aid of the discrete $H^1$-$L^\infty$-bound from Lemma \ref{lem:H1Linfty} in the appendix,
  \begin{align*}
    |\dq_\delta^+n\big|
    \le 2\sum_{k=1}^N\rho_k \|\phi_k\|_{L^\infty_\delta(\crc)}\big|\dq_\delta^+\phi_k\big|
    \le 2\sum_{k=1}^N\rho_k \big(1+\|\dq_\delta^+\phi_k\|_{L^2_\delta(\crc)}\big)\,\big|\dq_\delta^+\phi_k\big|.
  \end{align*}
  So finally, in combination with Minkowski's inequality, the bound \eqref{eq:phiH1up} implies that
  \begin{align*}
    \big\|\dq_\delta^+n\big\|_{L^2_\delta(\crc)}
    \le 4 \sum_{k=1}^N\rho_k \big(1+\|\dq_\delta^+\phi_k\|_{L^2_\delta(\crc)}\big)
    \le 4 + \frac{8}{\hbar^2}\big(\dent(n)-\underline{\dent}_{\hbar^2/2}\big).
  \end{align*}
  Now apply this to estimate to $n=n_\delta(T;\cdot)$, using that $\dent(n_\delta(T))\le\dent(n(0))$ and \eqref{eq:overdent}.
  This yields \eqref{est:mono}.
\end{proof}
For later reference, we draw some further conclusions from the considerations above.
\begin{lemma}
  There is a $\delta$-uniform constant $C$ such that for any $T>0$:
  \begin{align}
    \label{est:ninfty}
    \|n_\delta(T)\|_{\Lp\infty} &\le C\big(1+\overline\dent\big), \\
    \label{est:nebendiag}
    \|\nu_\delta^+-n_\delta\|_{\Lp\infty} &\le C\big(1+\overline\dent\big)\delta^{1/2}.
  \end{align}
\end{lemma}
\begin{proof}
  The first claim \eqref{est:ninfty} is a combination of the key estimate \eqref{est:mono} above with the the $H^1$-$L^\infty$-bound from Lemma \ref{lem:H1Linfty} in the appendix.
  For the proof of \eqref{est:nebendiag}, observe that at time $t=T$:
  \begin{align*}
    \nu^+(\xi)-n(\xi)
    = \sum_{k=1}^N \rho_k\big[\phi_k(\xi)\phi_k(\xi+\delta)-\phi_k(\xi)^2\big]
    = \delta\sum_{k=1}^N\rho_k\phi_k(\xi)\,\dq^+\phi_k(\xi).
  \end{align*}
  Directly from the definition of the $\Lp 2$-norm, it follows that that $\delta^{1/2}|\dq^+\phi_k(\xi)|\le\|\dq^+\phi_k\|_{\Lp 2}$. In combination with the $H^1$-$L^\infty$-bound form Lemma \ref{lem:H1Linfty}, we obtain that
  \begin{align*}
    \big|\nu^+(\xi)-n(\xi)\big|
    \le \delta^{1/2}\sum_{k=1}^N\rho_k\big(1+\|\dq^+\phi_k\|_{\Lp 2}\big)\,\|\dq^+\phi_k\|_{\Lp 2}
    \le 2\delta^{1/2}\sum_{k=1}^N\rho_k\big(1+\|\dq^+\phi_k\|_{\Lp 2}^2\big).
 \end{align*}
  By means of \eqref{eq:phiH1up}, the claim \eqref{est:nebendiag} follows.  
\end{proof}

\subsection{Convergence}
To define appropriate spatial interpolations, introduce the following families of hat functions $\Lambda_\delta\in C^{0,1}(\crc)$ and bump functions $\Phi_\delta\in C^{1,1}(\crc)$:
\begin{align*}
  \Lambda_\delta^\zeta(x)
  &=
  \begin{cases}
    1-(\zeta-x)/\delta & \text{for $\zeta-\delta\le x\le\zeta$}, \\
    1-(x-\zeta)/\delta & \text{for $\zeta\le x\le\zeta+\delta$}, \\
    0 & \text{otherwise},
  \end{cases}
  \\
  \Phi_\delta^\zeta(x) &= \delta^{-1}\int_{\zeta-\delta}^x\big(\Lambda_\delta^\zeta(x')-\Lambda_\delta^{\zeta+\delta}(x')\big)\dd x',
\end{align*}
for $\zeta\in\crcN$. By definition, we have that
\begin{align}
  \label{eq:Phidx}
  \partial_x\Phi_\delta^\zeta(x)
  = \big(\Lambda_\delta^\zeta(x)-\Lambda_\delta^{\zeta+\delta}(x)\big)/\delta
  = \big(\Lambda_\delta^\zeta(x)-\Lambda_\delta^\zeta(x+\delta)\big)/\delta
  = - \dq^+_\delta\Lambda_\delta^\zeta(x).
\end{align}
Notice further that $\Phi_\delta^\zeta$ is non-negative, is positive on $(\zeta-\delta,\zeta+2\delta)$, is bounded by one, and has integral equal to $\delta$.
Define the piecewise linear interpolation $\hat n_\delta$ of the density $n_\delta$ and the piecewise quadratic interpolation $F_\delta$ of the flux accordingly by
\begin{align}
  \label{eq:inter-n}
  \hat n_\delta(t;x) = \sum_\zeta n_\delta(t;\zeta)\Lambda_\delta^\zeta(x),
  \quad
  F_\delta(t;x) = \sum_\zeta \big(\nu_\delta^+\dq_\delta^+A_\delta\big)(t;\zeta) \,\Phi_\delta^\zeta(x),
\end{align}
\begin{lemma}
  \label{lem:F}
  For all $t>0$ and $x\in\crc$,
  \begin{align}
    \label{eq:nF}
    \partial_t\hat n_\delta(t;x) = \partial_xF_\delta(t;x).
  \end{align}
  Moreover, $F_\delta$ is $\delta$-uniformly bounded in $L^2\big((0,\infty)\times\crc\big)$:
  \begin{align}
    \label{eq:F22}
    \int_0^T\|F(s)\|_{L^2(\crc)}^2\dd s \le C(1+\sqrt\delta) \,(1+\overline\dent)^2,
  \end{align}
  with a $C$ that may depend on $T$ but not on $\delta$.
\end{lemma}
\begin{remark}
  \eqref{eq:nF} above can be obtained from the discrete nlQDD equation \eqref{eq:DQDD} by convolution in space with a rectangular mollifier supported on $[-\delta/2,\delta/2]$.  
\end{remark}
\begin{proof}
  By definition of $\hat n_\delta$ in \eqref{eq:inter-n} and by relation \eqref{eq:Phidx}, using a summation by parts, we obtain
  \begin{align*}
    \partial_t\hat n_\delta(t;x)
    = \sum_\zeta \dot n_\delta(t;\zeta)\,\Lambda_\delta^\zeta(x) 
    &= \sum_\zeta\dq_\delta^-\big(\nu_\delta^+\dq_\delta^+A\big)(t;\zeta)\,\Lambda_\delta^\zeta(x) \\
    &= -\sum_\zeta \big(\nu_\delta^+\dq_\delta^+A\big)(t;\zeta)\,\dq_\delta^+\Lambda_\delta^\zeta(x) \\
    &= \sum_\zeta \big(\nu_\delta^+\dq_\delta^+A\big)(t;\zeta)\,\partial_x\Phi_\delta^\zeta(x) 
   = \partial_xF_\delta(t;x).
  \end{align*}
  This proves \eqref{eq:nF}. 
  Next, since $\Phi_\delta^\zeta\Phi_\delta^{\zeta'}\equiv 0$ if $|\zeta-\zeta'|>2\delta$, one concludes --- using the elementary estimate $2ab\le a^2+b^2$ --- that
  \begin{align*}
    \|F\|_{L^2}^2
    &= \sum_{\zeta,\zeta'} \int_\crc\big(\nu_\delta^+\dq_\delta^+A\big)(t;\zeta)\Phi_\delta^\zeta(x)\, \big(\nu_\delta^+\dq_\delta^+A\big)(t;\zeta')\Phi_\delta^{\zeta'}(x)\dd x\\
    &\le 3\sum_\zeta \big(\nu_\delta^+\dq_\delta^+A\big)^2(t;\zeta)\,\left(\int_{\crc}\big(\Phi_\delta^\zeta(x)\big)^2\dd x\right) 
    \le 9 \|\nu_\delta^+(t)\|_{L^\infty_\delta(\crc)}\,\delta\sum_\zeta \nu_\delta^+\big(\dq_\delta^+A(t;\zeta)\big)^2,
  \end{align*}
  where we have used that $\Phi_\delta^\zeta$ is supported on an interval of length $3\delta$, with $0\le\Phi_\delta^\zeta\le1$, for the last estimate above.
  From here, recalling \eqref{est:ninfty}, \eqref{est:nebendiag} and \eqref{est:dissip}, we finally arrive at \eqref{eq:F22}.
\end{proof}
The first step in our convergence proof is to show uniform convergence of $\hat n_\delta$.
\begin{proposition}
  \label{prp:nconv}
  There are a function $n_*\in L^\infty\big((0,\infty);H^1(\crc)\big)\cap C^{1/10}_\loc(\setRnn\times\crc)$ such that  $\hat n_{\delta}\to n_*$ locally uniformly on $\setRnn\times\crc$ along a suitable sub-sequence $\delta\searrow0$.
\end{proposition}
\begin{proof}
  Below, we show that $\hat n_\delta$ is $\delta$-uniformly bounded in $C^{1/10}([0,T]\times\crc)$, for any $T>0$.
  The claim about uniform convergence then follows directly from the Arzela-Ascoli theorem.

  Pick $t\ge0$ and $\tau>0$ with $t+\tau\le T$. From the evolution equation \eqref{eq:nF}, we conclude that
  \begin{align*}
    \hat n_\delta(t+\tau)-\hat n_\delta(t)
    = \int_t^{t+\tau}\partial_t\hat n_\delta(s)\dd s = \partial_x\int_t^{t+\tau}F_\delta(s)\dd s.
  \end{align*}
  We apply the interpolation estimate from Lemma \ref{lem:C110} in the appendix:
  \begin{align*}
    \big\|\hat n_\delta(t+\tau)-\hat n_\delta(t)\big\|_{C^{1/10}(\crc)}
    &\le C\big\|\partial_x\hat n_\delta(t+\tau)-\partial_x\hat n_\delta(t)\big\|_{L^2(\crc)}^{4/5}
    \left\|\int_t^{t+\tau}F_\delta(s)\dd s\right\|_{L^2(\crc)}^{1/5} \\
    &\le C'\sup_{s\in[0,T]}\|\partial_x\hat n_\delta(s)\|_{L^2}^{4/5}\left(\int_t^{t+\tau}\|F_\delta(s)\|_{L^2(\crc)}\dd s\right)^{1/5} \\
    &\le C''(1+\overline\dent)^{4/5}\|F_\delta\|_{L^2([0,T]\times\crc)}^{1/5} h^{1/10}.
  \end{align*}
  Substitute \eqref{eq:F22} above, divide by $h^{1/10}$, and take the supremum with respect to $t$ and $\tau$.
  This shows boundedness of $\hat n_\delta$ in $C^{1/10}([0,T]\times\crc)$.
\end{proof}
The second step is the passage with the continuity equation \eqref{eq:nF} to the limit.
\begin{proposition}
  \label{prp:abstract}
  There is a $v_*\in L^2((0,\infty)\times\crc;n_*)$ such that $F_{\delta}\rightharpoonup n_*v_*$ weakly in $L^2((0,\infty)\times\crc)$ along a suitable sub-sequence $\delta\searrow0$, and
  \begin{align}
    \label{eq:abstract}
    \partial_tn_* = \partial_x(n_*v_*)
  \end{align}
  holds in the sense of distributions.
\end{proposition}
\begin{proof}
  Recalling Lemma \ref{lem:F}, one can select a sub-sequence $\delta\searrow0$ such that $F_{\delta}$ converges weakly in $L^2((0,\infty)\times\crc)$ to a limit $F_*$, and the linear relation \eqref{eq:nF} passes to limit, $\partial_tn_*=\partial_xF_*$. It therefore only remains to be shown that the weak limit $F_*$ factors in the given way. For that, we invoke the abstract convergence result from \cite[Theorem 5.4.4]{AGS}: let $T>0$ be given and consider
  \begin{align*}
    \mu_\delta(t;x) := \frac1{Z_{T,\delta}}\sum_\zeta \nu_\delta^+(t;\zeta)\Phi_\delta^\zeta(x),
  \end{align*}
  where the normalizing constant $Z_{T,\delta}$ is chosen such that $\mu_\delta$ is a probability density on $[0,T]\times\crc$. By the uniform convergence of $\hat n_{\delta}$ to $n_*$, and by the proximity \eqref{est:nebendiag} of $\nu_\delta^+$, it follows that $\mu_\delta$ converges in $L^1([0,T]\times\crc)$ to $\mu_*=T^{-1}n_*$. Accordingly, we define the velocity field
  \begin{align*}
    v_\delta(t;x) = \frac{F_\delta(t;x)}{\mu_\delta(t;x)} = Z_{T,\delta} \frac{\sum_\zeta \big(\nu_\delta^+\dq_\delta^+A\big)(t;\zeta)\Phi_\delta^\zeta(x)}{\sum_\zeta \nu_\delta^+(t;\zeta)\Phi_\delta^\zeta(x)}.
  \end{align*}
  At each fixed $(t;x)\in[0,T]\times\crc$, we have by the Cauchy-Schwarz inequality for sums that
  \begin{align*}
    |v_\delta(t;x)|^2\mu_\delta(t;x)
    = Z_{T,\delta} \frac{\sum_\zeta \big(\nu_\delta^+\dq_\delta^+A\big)(t;\zeta)\Phi_\delta^\zeta(x)}{\sum_\zeta \nu_\delta^+(t;\zeta)\Phi_\delta^\zeta(x)}
    \le Z_{T,\delta} \sum_\zeta \nu_\delta^+(t;\zeta)\,\big[\dq_\delta^+A(t;\zeta)\big]^2\Phi_\delta^\zeta(x).
  \end{align*}
  Arguing as in the proof of Proposition \ref{prp:nconv} above, we obtain from \eqref{est:dissip} the bound
  \begin{align*}
    \int_0^T\int_\crc |v_\delta|^2\mu_\delta\dd x\dd t
    \le CZ_{T,\delta} \int_0^T\sum_\zeta \nu_\delta^+(t;\zeta)\,\big[\dq_\delta^+A(t;\zeta)\big]^2\dd t
    \le C'Z_{T,\delta}(1+\overline\dent).
  \end{align*}
  The hypotheses of \cite[Theorem 5.4.4]{AGS} are met, an application of that result yields the desired factorization.
\end{proof}
In the third and final step, we identify the velocity field $v_*$ in \eqref{eq:abstract}.
In analogy to $\hat n_\delta$, introduce the linear interpolation $\hat A_\delta$ of $A_\delta$, that is
\begin{align*}
  \hat A_\delta(t;x) = \sum_\zeta A_\delta(t;\zeta)\Phi_\delta^\zeta(x).
\end{align*}
\begin{proposition}
  \label{prp:theendofQDD}
  In addition to the hypotheses of Proposition \ref{prp:abstract}, assume that $n_*$ is positive on $[0,T]\times\crc$. Then the linearly interpolated potentials $\hat A_\delta$ converge weakly in $L^2([0,T];H^1(\crc))$ to a limit $A_*$ as $\delta\searrow0$, after passing to a further sub-sequence if necessary. That limit satisfies the constitutive relation $n_*=\nop[A_*]$, and the velocity field in \eqref{eq:abstract} is given by $v_*=\partial_xA_*$.
\end{proposition}
\begin{proof}
  Since $n_*$ is continuous on $[0,T]\times\crc$, positivity implies the existence of a positive lower bound $\underline n>0$ such that $n_*\ge3\underline n$. Since $n_\delta$ converges uniformly to $n_*$ on $[0,T]\times\crc$, we have $n_\delta\ge2\underline n$ for all sufficiently small  $\delta>0$. It now follows from estimate \eqref{est:nebendiag} that $\nu_\delta^+\ge\underline n$ for all sufficiently small $\delta>0$. The dissipation estimate \eqref{est:dissip} then implies that $\dq_\delta^+A_\delta$ is $\delta$-uniformly bounded in $L^2([0,T];L^2_\delta(\crc))$. At almost every $t\in[0,T]$ we have that
  \begin{align*}
    \max A_\delta(t;\cdot) - \min A_\delta(t;\cdot) \le \frac\delta2\sum_\xi \big|\dq_\delta^+A_\delta(t;\xi)\big| \le \frac12\big\|\dq_\delta^+A(t;\cdot)\big\|_{L^2(\crc)}.
  \end{align*}
  In combination with \eqref{eq:Aupdown}, we obtain
  \begin{align*}
    \big\|A_\delta(t;\cdot)\big\|_{L^\infty_\delta(\crc)} \le \log Z_{\hbar^2,\delta} + \frac12\big\|\dq_\delta^+A(t;\cdot)\big\|_{L^2(\crc)}.
  \end{align*}
  Since $\log Z_{\hbar^2,\delta}\le \sqrt{\pi/\hbar^2}/4$, we conclude that $\hat A_\delta$ is uniformly bounded in $L^2([0,T];H^1(\crc))$. Thus $\hat A_\delta$ converges weakly to some $A_*$ in that space, along a suitable sub-sequence $\delta\searrow0$.
  
  Next, we conclude that $n_*=\nop[A_*]$: it is easily seen that the (spatially piecewise constant) functions $A_\delta$ converge weakly in $L^2((0,T)\times\crc)$ to the same limit $A_*$. Hence, we can invoke the identification result in Corollary \ref{cor:goal}; note that the uniform convergence of $n_\delta$ to $n_*$ obtained in Proposition \ref{prp:nconv} clearly implies strong convergence in $L^2((0,T)\times\crc)$. 

  Finally, we show that the limiting velocity field $v_*$, implicitly defined in \eqref{eq:abstract}, see Proposition \ref{prp:abstract}, equals $\partial_xA_*$. Note that we are only concerned about the identification of the limit, convergence of the $F_\delta$ follows from Lemma \ref{lem:F}. Since $n_\delta$ converges uniformly to $n_*$, also $\nu_\delta^+$ converges to $n_*$ due to the estimate \eqref{est:nebendiag}. By weak convergence of $\partial_x\hat A_\delta$ to $\partial_x A_*$, also $\dq_\delta^+A_\delta$ converges to $\partial_xA_*$, weakly in $L^2([0,T]\times\crc)$. In conclusion, $\nu_\delta^+ \dq_\delta^+A_\delta\rightharpoonup n_*\,\partial_xA_*$ in $L^1((0,T)\times\crc)$. Now let $\varphi\in C([0,T]\times\crc)$ be a continuous test function. We wish to show that the following two integrals $I_\delta$ and $J_\delta$ converge to the same limit:
  \begin{align*}
    I_\delta &= \int_0^T\int_\crc \nu^+_\delta\dq^+_\delta A_\delta\,\varphi\dd x\dd t
   =  \int_0^T \delta \sum_\zeta \nu_\delta^+(t;\zeta)\dq_\delta^+A(t;\zeta) \left(\fint_{I_\zeta}\varphi(t;x) \dd x\right) \dd t \\
    J_\delta &= \int_0^T\int_\crc F_\delta\,\varphi\dd x\dd t
   =  \int_0^T \delta \sum_\zeta \nu_\delta^+(t;\zeta)\dq_\delta^+A(t;\zeta) \left(\frac1\delta\int_\crc\Phi^\zeta_\delta(t;x)\varphi(t;x) \dd x\right) \dd t .
  \end{align*}
  Since $\Phi^\zeta$ integrates to $\delta$, since $\varphi$ is continuous, and since $\nu_\delta^+\dq_\delta^+A_\delta$ is $\delta$-uniformly bounded in $L^1([0,T]\times\crc)$, the limit of $J_\delta$ must indeed coincide with that of $I_\delta$, which is $\int_0^T\int_\crc n_*\partial_xA_*\,\varphi\dd x\dd t$.
\end{proof}
Theorem \ref{thm:3} now follows from the above: Propositions \ref{prp:nconv} and \ref{prp:theendofQDD} yield uniform convergence $n_\delta\to n_*$ and weak convergence $A_\delta\to A_*$, respectively, and Propositions \ref{prp:abstract} and \ref{prp:theendofQDD} yield in combination that $n_*$ and $A_*$ satisfy the nlQDD equation \eqref{eq:QDD}

\begin{appendix}

  \section{Trace convexity}
  \begin{lemma}
    \label{lem:tracemono}
    Assume $\phi:J\to\setR$ is a continuous increasing function on the interval $J\subseteq\setR$, and $A,A'\in\setC^{N\times N}$ are self-adjoint with $A\le A'$ and $\sigma(A)\cup\sigma(A')\subset J$.
    Then $\tr{\phi(A)}\le\tr{\phi(A')}$.
  \end{lemma}
  \begin{proof}
    Assume $\phi$ is also differentiable, hence $\phi'\ge0$ by monotonicity. Since $B:=A'-A$ is positive semi-definite, we have by the fundamental theorem of calculus that
    \begin{align*}
      \tr{\phi(A')}-\tr{\phi(A)}
      & = \int_0^1\frac{\dn}{\dd t}\tr{\phi(A+tB)}\dd t \\
      &= \int_0^1\tr{\phi'(A+tB)B}\dd t
      = \int_0^1\tr{B^{1/2}\phi'(A+tB)B^{1/2}}\dd t.
    \end{align*}
    Since $\phi'(A+tB)$ is positive semi-definite, the integrand above is non-negative.
  \end{proof}
  \begin{proposition}
    \label{lem:carlen}
    If $\phi:J\to\setR$ is a convex function on the interval $J\subseteq\setR$, and $A,A'\in\setC^{N\times N}$ are self-adjoint with $A\le A'$ and $\sigma(A)\cup\sigma(A')\subset J$.
    Then the map $[0,1]\ni s\mapsto \tr{(1-s)A+sA'}$ is convex.
    Moreover, Klein's inequality holds:
    \begin{align}
      \label{eq:klein}
      \tr{\phi(A')-\phi(A)-(A'-A)\phi'(A)} \ge 0.
    \end{align}
  \end{proposition}
  \begin{proof}
    These results are proven e.g. in \cite{carlenbook}. The convexity property and Klein's inequality, respectively, are contained in Theorems 2.9 and 2.10 therein.
  \end{proof}
  \begin{corollary}
    \label{cor:exppos}
    For all self-adjoint $A,A'\in\setC^{N\times N}$,
    \begin{align}
      \label{eq:exppos}
      \tr{\big(\exp(A')-\exp(A)\big)(A'-A)}\ge 0.
    \end{align}
  \end{corollary}
  \begin{proof}
    Consider the convex function $\phi:\setR\to\setR$ with $\phi(r)=\exp(r)$, with $\phi'(r)=\exp(r)$.
    Klein's inequality \eqref{eq:klein} yields
    \begin{align*}
      -\tr{(A'-A)\exp(A)}\ge-\tr{\exp(A')-\exp(A)}.
    \end{align*}
    Interchange the roles of $A$ and $A'$, and add the two inequalities to obtain \eqref{eq:exppos}.
  \end{proof}
  \begin{corollary}
    \label{lem:kleincor}
    For all $R,S\in\dens$,
    \begin{align}
      \label{eq:kleincor}
      \tr{(S-R)(\log S-\log R)} \ge \tr{(R-S)^2}.
    \end{align}
  \end{corollary}
  \begin{proof}
    Consider $\phi:[0,1]\to\setR$ given by
    \begin{align*}
      \phi(r) = r\log r-r-\frac12r^2, \quad \text{hence} \quad \phi'(r) = \log r - r,\quad \phi''(r)=\frac1r-1.
    \end{align*}
    Note that $\phi$ is obviously convex on its domain $[0,1]$ because $\phi''$ is positive there, and that $R,S\in\dens$ have their spectra in $[0,1]$.
    Klein's inequality \eqref{eq:klein} yields
    \begin{align*}
      0 &\le \tr{S\log S-R\log R-(S-R)\log R}-\tr{R-S}-\frac12\tr{R^2-S^2-2R(S-R)} \\
      &= \tr{S(\log R-\log S)} + \frac12\tr{(R-S)^2}.
    \end{align*}
    Now interchange the roles of $R$ and $S$, and add the two inequalities to conclude \eqref{eq:kleincor}.
  \end{proof}

\section{Interpolation estimates}
\begin{lemma}
  \label{lem:H1Linfty}
  Let $p\ge 1$.
  Every $\psi\in\Lp p$ with $\|\psi\|_{\Lp p}=1$ satisfies:
  \begin{align*}
    \|\psi\|_{\Lp\infty} \le 1+\|\dq^+\psi\|_{\Lp p}.
  \end{align*}
\end{lemma}
\begin{proof}
  By normalization of $\psi$, there is some $\xi_*\in \crcN$ such that $\psi(\xi_*)\le1$. And then, for any $\xi\in\crcN$,
  \begin{align*}
    |\psi(\xi)| \le |\psi(\xi_*)| + \delta\sum_{\xi'\in\crcN}|\dq^+\psi(\xi')| \le 1 +  \|\dq^+\|_{\Lp p},
  \end{align*}
  by Jensen's inequality.
\end{proof}
\begin{lemma}
  \label{lem:fisher}
  Let $n\in C(\crc)$ be positive and of finite Fisher information, $F:=\|\partial_x\sqrt n\|_{L^2(\crc)}<\infty$.
  For given $\delta=1/N$, define $n_\delta\in\Lp 1$ by averaging over the respective intervals $I_\xi$,
  \begin{align*}
    n_\delta(\xi) = \fint_{I_\xi} n\dd x.
  \end{align*}
  Then
  \begin{align}
    \label{eq:fisher}
    \delta\sum_\xi \big(\dq_\delta^+\sqrt{n_\delta}\big)^2 \le 8F^2.
  \end{align}  
\end{lemma}
\begin{proof}
  Fix some $\xi\in\crcN$.
  Since $n$ is continuous, there is some $\xi^*\in I_\xi$ such that $n(x_*)=n_\delta(\xi)$.
  It follows that 
  \begin{align*}
   \left|\sqrt n(\xi) - \sqrt n(x_*)\right|
    =\left| \int_{x_*}^{\xi}\partial_x\sqrt n\dd x\right|
    \le |\xi-x_*|^{1/2}\left(\int_{I_\xi}\big(\partial_x\sqrt n\big)^2\right)^{1/2}
    \le \delta \big\|\partial_x\sqrt n\big\|_{L^2(I_\xi)},
  \end{align*}
  and further that
  \begin{align*}
    \big|\dq_\delta^+\sqrt{n_\delta}(\xi)\big|
    &= \frac1{\delta}\left|\sqrt{n_\delta}(\xi+\delta)-\sqrt{n_\delta}(\xi)\right| \\
    &\le \frac1\delta\left(
      \left|\sqrt n(\xi+\delta) - \sqrt{n_\delta}(\xi+\delta)\right|
      + \left|\sqrt n(\xi+\delta)-\sqrt n(\xi)\right|
      + \left|\sqrt n(\xi) - \sqrt{n_\delta}(\xi)\right|
      \right) \\
    &\le \big\|\partial_x\sqrt n\big\|_{L^2(I_\xi\cup I_{\xi+\delta})}
      + \fint_\xi^{\xi+\delta}\big|\partial_x\sqrt n\big|\dd x
      \le 2 \big\|\partial_x\sqrt n\big\|_{L^2(I_\xi\cup I_{\xi+\delta})}.
  \end{align*}
  Now take the square and sum over $\xi\in\crcN$ to obtain \eqref{eq:fisher}.
\end{proof}
\begin{lemma}
  \label{lem:C110}
  There is a universal constant $C$ such that for any real $f\in H^1(\crc)$ of zero average,
  \begin{align}
    \label{eq:C110}
    \|f\|_{C^{1/10}} \le C \|\partial_xf\|_{L^2}^{4/5}\|F\|_{L^2}^{1/5},
  \end{align}
  where $F$ in any primitive of $f$, i.e., $\partial_xF=f$.
\end{lemma}
\begin{proof}
  First observe that
  \begin{align*}
    \|f\|_{L^2}^2 = \int_\crc (\partial_xF)^2\dd x 
    = - \int_\crc F\,\partial_{xx}F\dd x
    = -\int_\crc F\,\partial_xf\dd x
    \le \|F\|_{L^2}\|\partial_xf\|_{L^2}.
  \end{align*}
  Next, since $f$ is continuous and of zero average, there is some $x_*\in\crc$ with $f(x_*)=0$; by periodicity of $f$, we may assume $x_*=0$. Then, for any $x^*\in\crc$,
  \begin{align*}
    f(x^*)^2 = \int_0^{x^*} \partial_x(f^2)\dd x 
    = 2\int_0^{x^*}f\,\partial_xf\dd x
    \le 2\|f\|_{L^2}\|\partial_xf\|_{L^2}.
  \end{align*}
  In combination with the previous estimate, we obtain
  \begin{align}
    \label{eq:H1H-1}
    \|f\|_{L^\infty} \le \sqrt2\|\partial_xf|_{L^2}^{1/2}\|f\|_{L^2}^{1/2}\le\sqrt2\|\partial_xf\|_{L^2}^{3/4}\|F\|_{L^2}^{1/4}.
  \end{align}
  Let $x^*\in\crc$ and $h>0$ such that $x+h\in\crc$. Then
  \begin{align*}
    |f(x^*+h)-f(x^*)|
    \le \int_{x^*}^{x^*+h}\partial_xf\dd x \le h^{1/2}\|\partial_xf\|_{L^2},
  \end{align*}
  and therefore, using \eqref{eq:H1H-1},
  \begin{align*}
    |f(x^*+h)-f(x^*)|
    &= |f(x^*+h)-f(x^*)|^{4/5}|f(x^*+h)-f(x^*)|^{1/5} \\
    &\le \big(2\|f\|_{L^\infty}\big)^{4/5}h^{1/10}\|\partial_xf\|_{L^2}^{1/5}
      \le 2^{6/5} \|\partial_xf\|_{L^2}^{4/5}\|F\|_{L^2}^{1/5}h^{1/10}.
  \end{align*}
  Divide by $h^{1/10}$ and take the supremum with respect to $x^*$ and $h$ to obtain \eqref{eq:C110}.
\end{proof}

  \section{Properties of the discrete Laplacian}
  \label{app:Laplace}
  \begin{lemma}
    \label{lem:DeltaEV}
    Let $K=\{k_0,k_0+1,\ldots,k_0+N-1\}\subset\setZ$ be any set of $N$ consecutive integers. Then the eigenvalues and a corresponding basis of eigenvectors for $-\Delta_\delta$ are given by $(\omega_k)_{k\in J}$ and $(w_k)_{k\in K}$, respectively, with 
    \begin{align}
      \label{eq:DeltaEV}
      \omega_k=2\frac{1-\cos(2\pi k\delta)}{\delta^2},
      \quad
     \cmp{w_k}j = e^{2\pi i k\delta j}.
    \end{align}
  \end{lemma}
  \begin{proof}
    Simply observe that
    \begin{align*}
      \cmp{\Delta_\delta w_k}j
      = \frac{e^{2\pi ik\delta(j+1)}+e^{2\pi ik\delta(j-1)}-2 e^{2\pi ik\delta j}}{\delta^2}
      = \frac{\big(e^{2\pi ik\delta}+e^{-2\pi ik\delta}-2\big)e^{2\pi ik\delta j}}{\delta^2}
      = -\omega_k \cmp{w_k}j.
    \end{align*}
    The basis property of the $w_k$ is well-known.
  \end{proof}
  \begin{lemma}
    The numbers $\omega_k$ given in \eqref{eq:DeltaEV} satisfy
    \begin{align}
      \label{eq:DeltaEVbelow}
      \omega_k \ge 16k^2,
    \end{align}
    for all $k\in\setZ$ with $|k|\le N/2$. Consequently, for any $\alpha>0$ and $k_0\in\setZ$,
    \begin{align}
      \label{eq:LT}
      \sum_{k=k_0}^{k_0+N-1}e^{-\alpha\omega_k}\le 1 + \frac{\sqrt\pi}4 \alpha^{-1/2}.
    \end{align}
  \end{lemma}
  \begin{proof}
    For all $u\in\setR$ with $|u|\le\pi$,
    \begin{align*}
      \cos u \le 1-\frac2{\pi^2}u^2.
    \end{align*}
    If $|k|\le N/2$, then $\delta|k|\le \pi$, and therefore
    \begin{align*}
    \omega_k \ge \delta^{-2}\frac4{\pi^2}(2\pi k\delta)^2 \ge 16k^2,
    \end{align*}
    proving \eqref{eq:DeltaEVbelow}. Substitute this into the left-hand side of \eqref{eq:LT} to obtain
    \begin{align*}
      \sum_{k=k_0}^{k_0+N-1}e^{-\alpha\omega_k}
      \le \sum_{|k|\le N/2}e^{-\alpha\omega_k}
      \le 1 + 2\sum_{k=1}^\infty e^{-16\alpha k^2}
      \le 1 + 2\int_0^\infty e^{-16\alpha z^2}\dd z
      = 1 + \frac14\sqrt{\frac{\pi}{\alpha}},
    \end{align*}
    as had to be shown.
  \end{proof}

  \section{Discrete approximation of the heat kernel}
  \label{app:A}
  Here we prove Propostion \ref{prp:prekernel} on the approximation of the periodic heat kernel $\htkrnl$ by the spatially discrete ones $\htkrnl_\delta$,
  \begin{align*}
    \htkrnl(z) = \sum_{k\in\setZ} e^{-(2\pi k)^2\hbar^2 t}e^{2\pi ikz}, \quad
    \htkrnl_\delta(\zeta) = \sum_{|k|\lesssim N/2} e^{-\omega_k\hbar^2 t}e^{2\pi ikz},
  \end{align*}
  where $\omega_k$ are the eigenvalues of the discrete Laplacian, see \eqref{eq:DeltaEV}.
  First, we need two auxiliary elementary estimates.
  \begin{lemma}
    For each $p>0$,
    \begin{align}
      \label{eq:expest}
      K_p:=\sup_{z>0}z^pe^{-z} <\infty.
    \end{align}
  \end{lemma}
  \begin{proof}
    This follows since the non-negative and continuous function $z\mapsto z^pe^{-z}$ has limit zero for $z\searrow0$ and for $z\to\infty$, respectively.
  \end{proof}
  \begin{lemma}
    For any real $m\ge0$, there is a constant $M_m$ such that, for any $\alpha>0$, 
    \begin{align}
      \label{eq:gaussest}
      \sum_{k\in\setZ} |k|^me^{-\alpha k^2} \le M_m\alpha^{-(m+1)/2}e^{\alpha/4}.
    \end{align}
  \end{lemma}
  \begin{proof}
    We prove that    
    \begin{align}
      \label{eq:gausshelp}
      \sum_{k\in\setZ} |k|^me^{-\alpha k^2} \le 2^me^{\alpha/4}\int_{\setR} |x|^{m}e^{-\frac\alpha2\,x^2}\dd x\,;
    \end{align}
    the claim \eqref{eq:gaussest} then follows by direct evaluation of the right-hand side, with
    \begin{align*}
      M_m := 2^m\int_{\setR} |z|^me^{-z^2/2}\dd z.
    \end{align*}
    For the proof of \eqref{eq:gausshelp}, it suffices to show that
    \begin{align}      
      \label{eq:gausshelp2}
      |k|^me^{-\alpha k^2} \le 2^me^{\alpha/4}\int_{k-\frac12}^{k+\frac12} |x|^{m}e^{-\alpha/2\,x^2}\dd x,
    \end{align}
    for each $k\in\setZ$. For $k=0$ and $m>0$, \eqref{eq:gausshelp2} is trivially true, since the left-hand side is zero. For $m=0$, it reduces to
    \begin{align*}
      1 \le e^{\alpha/4}\int_{-1/2}^{1/2}e^{-\alpha/2\,x^2}\dd x,
    \end{align*}
    which is also true, because the integrand is larger or equal to $e^{-\alpha/8}$ on the domain of integration. Now let $m\ge0$ and $k\ge1$ be arbitrary. Observe that
    \begin{align*}
      |x|^{m}e^{-\frac\alpha2\,x^2}
      \ge \left(k-\frac12\right)^me^{-\frac\alpha2\,(k+\frac12)^2}
      \ge \left(\frac k2\right)^me^{-\frac\alpha2(k^2+k+\frac14)}
    \end{align*}
    for all $x\in(k-\frac12,k+\frac12)$. And now, since
    \begin{align*}
      \frac\alpha2(k^2+k+\frac14) = \alpha k^2 + \frac\alpha 4 - \frac\alpha2(k^2-\frac12)^2 \le  \alpha k^2 + \frac\alpha 4,
    \end{align*}
    the estimate \eqref{eq:gausshelp2} follows directly.
  \end{proof}
  Next, we establish a \emph{pointwise} comparison between the original and the discretized kernels. 
  \begin{lemma}
    \label{lem:prekernel1}
    There is a constant $A_1>0$ such that
    \begin{align}
      \label{eq:prekernel1}
      \max_{\xi\in\crcN}\big|\htkrnl^t(\xi)-\htkrnl_\delta^t(\xi)\big|\le A_1\delta^{1/4}t^{-3/4}.
    \end{align}
    The constant $A_1$ is uniform in $N\in\{2,3,\ldots\}$ and $t>0$.  
\end{lemma}
\begin{proof}[Proof of Lemma \ref{lem:prekernel1}]
  By Taylor expansion, for every $u\in\setR$,
  \begin{align*}
    0\le u^2-2(1-\cos u) \le \frac{u^4}{12}.
 \end{align*}
  This implies
  \begin{align*}
    0\le (2\pi k)^2-\omega_k\le \frac{(2\pi k\delta)^4}{12\delta^2} = Lk^4\delta^2 \quad\text{with}\quad L:=\frac{4\pi^4}{3}.
  \end{align*}
  In particular, for ``relatively small'' frequencies $|k|\le \delta^{-1/2}$, it follows that
  \begin{align*}
    0\le (2\pi k)^2-\omega_k\le \min{L,Lk^2\delta},
  \end{align*}
  and so, for those $k$, and for every $t\in(0,1]$,
  \begin{align}
    \label{eq:thelp1}
    0\le e^{((2\pi k)^2-\omega_k)\hbar^2 t} -1 \le L\hbar^2 e^{L\hbar^2}\delta k^2t.
  \end{align}
  For fixed $\xi\in\crcN$, 
  \begin{align}
    \nonumber
    \big|\htkrnl^t(\xi)-\htkrnl_\delta^t(\xi)\big|
    &= \left|
      \left[\sum_{|k|\le\delta^{-\frac12}}+\sum_{\delta^{-\frac12}<|k|}\right]e^{-(2\pi k)^2\hbar^2 t}e^{2\pi ik\xi}
      - \left[\sum_{|k|\le\delta^{-\frac12}}+\sum_{\delta^{-\frac12}<|k|\lesssim N/2}\right]e^{-\omega_k\hbar^2 t}e^{2\pi ik\xi}
      \right| \\
    \label{eq:t1}
    &\le \sum_{|k|\le\delta^{-\frac12}}\big|e^{-(2\pi k)^2\hbar^2 t}-e^{-\omega_k\hbar^2 t}\big| \\
    \label{eq:t2}
    &\qquad + \sum_{\delta^{-\frac12}<|k|\lesssim N/2} e^{-\omega_k\hbar^2 t} \\
    \label{eq:t3}
    &\qquad + \sum_{|k|>\delta^{-\frac12}} e^{-(2\pi k)^2t}.    
  \end{align}
  We treat the three terms \eqref{eq:t1}--\eqref{eq:t3} separately.
  For the first, we use \eqref{eq:thelp1} in combination with \eqref{eq:gaussest}:
  \begin{align*}
    \eqref{eq:t1} &= \sum_{|k|\le\delta^{-\frac12}}e^{-(2\pi k)^2\hbar^2 t}\big|1-e^{((2\pi k)^2-\omega_k)\hbar^2 t}\big| \\
    &\le L\hbar^2 e^{L\hbar^2}\sum_{|k|\le\delta^{-\frac12}}(\delta k^2t)e^{-(2\pi)^2\hbar^2 t\ k^2} \\
    &\le L\hbar^2 e^{L\hbar^2}\, M_2e^{\pi^2\hbar^2} \big(4\pi^2\hbar^2 t\big)^{-3/2}\delta t
      = (2\pi)^{-3}M_2L\hbar^{-1} e^{(L+\pi^2)\hbar^2} \,\delta t^{-1/2}.
  \end{align*}
  The second term is estimated with the help of \eqref{eq:DeltaEVbelow}, in combination with \eqref{eq:expest} and \eqref{eq:gaussest}. Using that for $k\in\setZ$ with $k>\delta^{-\frac12}$, we have that
  \[ k^2 = \big((k-\delta^{-\frac12})+\delta^{-\frac12}\big) \ge (k-\delta^{-\frac12})^2+\delta^{-1}, \]
  it follows with $\delta<1$ that
  \begin{align*}
    \eqref{eq:t2}
    \le \sum_{|k|>\delta^{-\frac12}} e^{-16k^2\hbar^2 t} 
    & \le e^{-16\delta^{-1}\hbar^2 t}\sum_{\ell\in\setZ}e^{-16\hbar^2 t\ \ell^2} 
    \le e^{-16\delta^{-1}\hbar^2 t}\;M_0e^{4\hbar^2 t}(16\hbar^2 t)^{-1/2} \\
    &\le \frac{M_0}4 (\hbar^2 t)^{-3/4}\delta^{1/4}\ (\hbar^2 t/\delta)^{1/4}e^{-\hbar^2 (t/\delta)}
      \le \frac{M_0K_{1/4}}{4\hbar} \delta^{1/4}t^{-3/4}.
  \end{align*}
  The third term is estimated like the second, the only difference is the prefactor $(4\pi)^2$ instead of $16$ in the exponent. We obtain accordingly
  \begin{align*}
    \eqref{eq:t3}
    \le \frac{M_0K_{1/4}}{4\pi\hbar} \delta^{1/4}t^{-3/4}.
  \end{align*}
  Summarizing the estimates on \eqref{eq:t1}--\eqref{eq:t3}, the claim \eqref{eq:prekernel1} follows.
\end{proof}
We are finally in the position to prove Propoisition \ref{prp:prekernel}.
\begin{proof}[Proof of Proposition \ref{prp:prekernel}]
  With estimate \eqref{eq:prekernel1} at hand, it suffices to prove the proximity of the values of $\htkrnl$ to the respective local averages. For each $\xi\in\crcN$, we have that
  \begin{align*}
    \fint_{I_\xi}\htkrnl^t(z)\dd z
    &= \delta^{-1}\int_{\xi-\delta/2}^{\xi+\delta/2} \sum_{k\in\setZ}e^{-(2\pi k)^2\hbar^2 t}e^{2\pi ikz}\dd z \\
    &= \sum_{k\in\setZ}e^{-(2\pi k)^2\hbar^2 t}\frac{e^{2\pi ik(\xi+\delta/2)}-e^{ik(\xi-\delta/2)}}{ik\delta}
    = \sum_{k\in\setZ}e^{-(2\pi k)^2\hbar^2 t}\frac{\sin(\pi k\delta)}{\pi k\delta}e^{2\pi ik\xi}.
  \end{align*}
  Once again, we use an elementary estimate:
  \begin{align*}
    0\le 1-\frac{\sin u}{u} \le \min(2,u^2/4) \le 2|u|^{1/2}
    \quad\text{  for all $u\in\setR$}.
  \end{align*}
  In combination with \eqref{eq:gaussest}, we conclude that  
  \begin{align*}
    \left|\fint_{I_\xi}\htkrnl^t(z)\dd z - \htkrnl^t(\xi)\right|
    &\le \sum_{k\in\setZ}e^{-(2\pi k)^2\hbar^2 t}\left|\frac{\sin(\pi k\delta)}{\pi k\delta}-1\right| \\
    &\le 2(\pi\delta)^{1/2}\sum_{k\in\setZ}|k|^{1/2} e^{-(2\pi)^2\hbar^2 t\ k^2} \\
    &\le 2\pi^{1/2}M_{1/2}\big((2\pi)^2\hbar^2 t\big)^{-3/4}\delta^{1/2}
      \le \frac{M_{1/2}}{\pi\hbar^{3/2}}\delta^{1/4}t^{-3/4},
  \end{align*}
  where we have used that $\delta\in(0,1)$.
\end{proof}

\end{appendix}


\bibliographystyle{plain}
\bibliography{biblio}

\end{document}